\documentclass[preprint,12pt]{elsarticle}

\usepackage{a4wide,amsfonts, amsmath, amscd}
\usepackage[psamsfonts]{amssymb}

\usepackage{graphicx}
\usepackage[cp1251]{inputenc}
\usepackage{amssymb}
\usepackage{amsthm}
\usepackage{cmap}

\usepackage{amsmath}

\biboptions{sort&compress}

\newcommand{\sysn}{\left\{\begin{array}{rcl}}
\newcommand{\sysk}{\end{array}\right.}

\renewcommand{\le}{\leqslant}
\renewcommand{\ge}{\geqslant}

\newtheorem{theorem}{Theorem}[section]
\newtheorem{lemma}[theorem]{Lemma}

\theoremstyle{example}

\newtheorem{proposition}[theorem]{Proposition}
\theoremstyle{definition}
\newtheorem{definition}[theorem]{Definition}
\newtheorem{remark}[theorem]{Remark}
\newtheorem{corollary}[theorem]{Corollary}

\journal{...}

\begin{document}

\title{Embedding Theorems for function spaces}
\author[affil1]{Mikhail Al'perin}

\address[affil1]{Krasovskii Institute of Mathematics and Mechanics, Yekaterinburg, Russia}

\ead{alper@mail.ru}

\author[affil1]{Sergei Nokhrin}

\address[affil2]{Ural Federal
 University, Yekaterinburg, Russia}

\ead{varyag2@mail.ru}

\author[affil1,affil2]{Alexander V. Osipov}


\ead{OAB@list.ru}

\begin{abstract} In this paper, we have proved results similar to Tychonoff's
Theorem on embedding a space of functions with the topology of
pointwise convergence into the Tychonoff product of topological
spaces, but applied to the function space $C(X,Y)$ of all
continuous functions from a topological space $X$ into a uniform
space $Y$ with the topology of uniform convergence on a family of
subsets of $X$ and with the (weak) set-open topology. We also
investigated the following question: how the topological embedding
of the space $C(X,Y)$ is related to algebraic structures (such as
topological groups, topological rings and topological vector
spaces) on $C(X,Y)$.
\end{abstract}


\begin{keyword}
embedding \sep function space  \sep set-open topology \sep
topology of uniform convergence \sep uniform space \sep
topological group \sep topological ring \sep topological vector
space

\MSC[2010] 54C35 \sep 54C40 \sep 54C25 \sep 54H11

\end{keyword}

\maketitle 


\section{Introduction}

The space of continuous functions from one topological space to
another is one of the main object studied in functional analysis
and topology. As a development of the concepts of a uniformly
converging sequence of functions and a pointwise converging
sequence of functions, the concepts of topologies of pointwise and
uniform convergence on the space of functions arose. Considerable
progress has been made in the study of these topologies. The
case of the topology of uniform convergence, this progress is due
to advances in the study of Banach spaces. And in the case of the
topology of pointwise convergence, most of the interesting results
are related to Tychonoff's Theorem \cite{tych} on the embedding of
a space of functions with such topology into Tychonoff's product
of topological spaces.

Later, other topologies began to be defined and studied at the
space $C(X,Y)$ of all continuous functions from a topological
space $X$ to a topological space $Y$. In 1945, Fox \cite{fox}
defined the compact-open topology, later Arens and Dugungi
\cite{AreDug} defined the set-open topology. In 1955, the concept
of the topology of uniform convergence on a family of subsets
first appeared in \cite{kell}. Over the past twenty years,
significant progress has been made in the study of
topological-algebraic properties of the space $C(X,Y)$ with the
(weak) set-open topology [6,7,17,24] and [26-36].

In this paper, we study a construction similar to Tychonoff's
Theorem, but applying for a function space $C(X,Y)$ with the
topology of uniform convergence on a family of subsets of $X$ and
with the (weak) set-open topology. The motivation is the following
two theorems.

\begin{theorem}(Theorem 3.4.11 in \cite{Eng}) If $X$ is a
$k$-space, then for every topological space $Y$ the space $C(X,Y)$
with the compact-open topology (with the topology of pointwise
convergence) is homeomorphic to the limit of the inverse system
$S(X)=\{C(K,Y), \pi^{K_1}_{K_2}, \mathcal{K}(X)\}$ of the spaces
$C(K,Y)$ with the compact-open topology (with the topology of
pointwise convergence) where $K\in \mathcal{K}(X)$ and
$\mathcal{K}(X)$ denote the family of all non-empty compact
subsets of a Hausdorff space $X$.
\end{theorem}

\begin{theorem}(\cite{alno}) For a Hausdorff space $X$ there exists a
$k$-space $X_k$ such that $C_c(X,Y)$ can be embedded in
$C_c(X_k,Y)$ for any topological space $Y$.
\end{theorem}

In this paper we study two natural questions: How essential is the
condition to be a $k$-space in these theorems? How  the
topological embedding of the space $C(X,Y)$ is related to
algebraic structures (such as topological groups, topological
rings and topological vector spaces) on $C(X,Y)$?

\section{Notation and terminology}
 The set of positive integers is denoted by $\mathbb{N}$ and
$\omega=\mathbb{N}\cup \{0\}$. Let $\mathbb{R}$ be the real line,
we put $\mathbb{I}=[0,1]\subset \mathbb{R}$, and let $\mathbb{Q}$
be the rational numbers. We denote by $\overline{A}$ (or $Cl_X A$)
the closure of $A$ (in $X$).

As in \cite{kawi,fran}, a {\it natural cover} $\Sigma$ assigns to
each topological space $X$ a cover $\Sigma_X$ of $X$ satisfying:

(i) if $S\in \Sigma_X$ and $S$ is homeomorphic to $T\subseteq Y$,
then $T\in \Sigma_Y$,

(ii) if $f:X\rightarrow Y$ is continuous and $S\in\Sigma_X$, there
is a $T\in\Sigma_Y$ with $f(S)\subseteq T$.

\begin{definition}\label{t1_2_1}{\rm\cite{kell}}
Let $X$ be a topological space, $\lambda\subseteq 2^X$, $(Y,\mu)$
be a uniform space. A topology on $C(X,Y)$ generated by the
uniformity
$$\nu=\{\langle A,M\rangle \subseteq C(X,Y)\times
C(X,Y):A\in\lambda,M\in\mu\}$$ where
$$\langle A,M\rangle = \{\langle f,g\rangle \in C(X,Y)\times
C(X,Y):\forall x\in A~\langle f(x),g(x)\rangle\in M\}$$ is called
{\it topology of uniform convergence on elements of $\lambda$} and
denote by $C_{\lambda,\mu}(X,Y)$.
\end{definition}

The well-known fact that if $\lambda$ is a family of all compact
subsets of $X$ or all finite subsets of $X$ then the topology on
$C(X,Y)$ induced by the uniformity $\nu$ of uniform convergence on
elements of $\lambda$ depends only on the topology induced on $Y$
by the uniformity $\mu$ (see \cite{kell, Eng}). In these cases, we
will use the notation $C_c(X,Y)$ and $C_p(X,Y)$, respectively. If
$Y=\mathbb{R}$ then $C_c(X)$ and $C_p(X)$, respectively.

In case, if $(Y,\rho)$ is a metric space and the uniformity $\mu$
is induced by the metric $\rho$, then for $C_{\lambda,\mu}(X,Y)$,
we will use the notation $C_{\lambda,\rho}(X,Y)$ and
$C_{\lambda,\rho}(X)$ for the case $Y=\mathbb{R}$.

If $X\in\lambda$, we write $C_\mu(X,Y)$ in place of
$C_{\lambda,\mu}(X,Y)$ and $C_{\mu}(X)$ in place of
$C_{\mu}(X,\mathbb{R})$.

\begin{remark}\label{t1_2_2}

For the topology of uniform convergence on elements of $\lambda$,
we assume that the following natural conditions are satisfied.

(1) if $A\in\lambda$ and $A'\subseteq A$ then $A'\in\lambda$.

(2) if $A_1, A_2\in\lambda$ then $A_1\bigcup A_2\in\lambda$.

(3) if $A\in\lambda$ then $\overline{A}\in\lambda$.

 Note that $C_{\lambda,\mu}(X,Y)$ is Hausdorff if and only if the set $\bigcup \lambda$ is dense in $X$,
 i.e.,
\hbox{$\overline{\bigcup\{A:A\in\lambda\}}=X$} (see
\cite{Bur75b}).

 Then we have an additional condition on $\lambda$.

(4) $\lambda$ is a cover of $X$.
\end{remark}

The {\it set-open topology} on a family $\lambda$ of non-empty
subsets of the set $X$ is a generalization of the compact-open
topology and of the topology of pointwise convergence. This
topology, first introduced by Arens and Dugunji in \cite{AreDug},
is one of the important topologies on $C(X,Y)$.

If $A\subseteq X$ and $V\subseteq Y$, then $[A,V]$ is defined by
as $[A,V]=\{f\in C(X,Y): f(A)\subseteq V\}$.

\begin{definition}\label{t1_2_12}
Let $X$ and $Y$ be topological spaces, $\lambda\subseteq 2^X$ . A
topology on $C(X,Y)$ is called a {\it $\lambda$-open topology}
(set-open topology) provided the family $\{[A,V]: A\in \lambda$
and $V$ is open in $Y \}$ form a subbase for the topology. The
function space $C(X,Y)$, provided with this topology, is denoted
by $C_{\lambda}(X,Y)$.
\end{definition}

The {\it weak set-open} topology on a family $\lambda$ of nonempty
subset of the set $X$ (the $\lambda^*$-open topology) is a
generalization of the bounded-open topology \cite{kundu} and of
the topology of pointwise convergence. All sets of the form
$[F,U]^*=\{f\in C(X,Y): \overline{f(F)}\subseteq U\}$, where $F\in
\lambda$ and $U$ is an open subset of $Y$, form a subbase of the
$\lambda^*$-open topology. The function space $C(X,Y)$ provided
with this topology is denoted by $C_{\lambda^*}(X,Y)$.

\section{$\lambda$- and $\lambda_f$-leaders}

In this section we introduce and study the notation of a {\it
$\lambda$-space} as a generalization of the notation of a
$k$-space first defined by Hurewicz. The class of $k$-spaces was
introduced by Gale in \cite{gal} (where the notion is ascribed to
Hurewicz).

A topological space $X$ is called a {\it $k$-space} if $X$ is a
Hausdorff space and $X$ is an image of a locally compact space
under a quotient mapping. In other words, $k$-spaces are Hausdorff
spaces that can be represented as quotient spaces of locally
compact spaces.

Examples of $k$-spaces are Hausdorff spaces which are locally
compact, or satisfy the first axiom of countability. Hence all
metric spaces are $k$-spaces.

Recall that a Hausdorff space is a $k$-space if it has the {\it
final topology} with respect to all inclusions $C \rightarrow X$
of compact subspaces $C$ of $X$, so that a set $A$ in $X$ is
closed in $X$ if and only if $A\cap C$ is closed in $C$ for all
compact subspaces $C$ of $X$.

\begin{definition} Let $X$ be a Hausdorff topological space, $\lambda\subseteq 2^X$.
The space $X$ is called a {\it $\lambda$-space}, if for any
$A\subseteq X$, $A$ is closed in $X$ if and only if $A\cap
\overline{B}$ is closed in $\overline{B}$ for all $B\in \lambda$.
\end{definition}

In the case when $\lambda$ is a natural cover of $X$, the notion
of $\lambda$-space was introduced by S.P. Franklin \cite{fran}.

Note that if $X$ is a $\lambda$-space, then $X\setminus
\overline{\bigcup\{A:\, A\in\lambda\}}$ either empty or discrete.

\begin{definition}
Let $X$ be a topological space, $\lambda\subseteq 2^X$. The space
$X$ is a {\it locally $\lambda$-space} if for any $x\in X$ there
are a neighbourhood $U$ of $x$ and an $A\in \lambda$ such that
$\overline{U}=\overline{A}$.
\end{definition}

\begin{theorem} Let $X$ be a Hausdorff space,
$\lambda$ be a cover of $X$. Then the following assertions are
equivalent:

(1)  $X$ is a $\lambda$-space;

(2) there are a space $Y$ and $\gamma\subseteq 2^Y$ such that $Y$
is a locally $\gamma$-space, $X$ is a quotient image of $Y$ with
respect to a quotient mapping $f:\, Y\to X$ and
$\gamma\subseteq\{B\subseteq Y:\ \exists A\in\lambda\ \mbox{and}\
f(B)=\overline{A}\}$.
\end{theorem}

\begin{proof} Let $f:\, Y\to X$ be a quotient map and  $Y$ be a locally $\lambda$-space. Let $S\subseteq X$ be a set such that $S\cap \overline{A}$
is closed in $\overline{A}$ for all $A\in\lambda$. Take a point
$y\in \overline{f^{-1}(S)}$ and a neighbourhood $U(y)$ of $y$ such
that $\overline{U(y)}=\overline{B}$ for some $B\in\gamma$. The set
$f^{-1}(S\cap \overline{f(U(y)})=f^{-1}(S\cap \overline{A})$
where $A\in\lambda$, and, hence, it is closed in $Y$. On the other
hand, $f^{-1}(S)\cap \overline{U(y)}\subseteq f^{-1}(S\cap
f(\overline{U(y)}))\subseteq f^{-1}(S)$. The point $y\in
\overline{f^{-1}(S)\cap \overline{U(y)}}$, hence, $y\in
f^{-1}(S)$. Then, the set $f^{-1}(S)$ is closed in $Y$. Since the
mapping $f$ is quotient, $S$ is closed.

Let $X$ be a Hausdorff $\lambda$-space and
$\bigcup\{A:\,A\in\lambda\}=X$.

The space $\tilde X=\bigoplus\{\overline{A}:\,A\in\lambda\}$ is a
locally $\gamma$-space where $\gamma$ is a family of members of
direct sum of elements of $\lambda$. Consider the family of maps
$\{i_A:\, A\in\lambda\}$ where $i_A:\, \overline{A}\to X$ is an
identification continuous map. Define the mapping $f:\, \tilde
X\to X$ as $f|_{\overline{A}}=i_A$ for every $A\in\lambda$. Since
for any $A_1,\, A_2\in\lambda$ the equality
$f_{A_2}|_{\overline{A_1}\cap\overline{A_2}}=f_{A_1}|_{\overline{A_1}\cap\overline{A_2}}$
holds, the mapping $f$ is well-defined and continuous.

Note that for a subset $S$ of $X$, if $S\cap \overline{A}$ is
closed in $\overline{A}$  for each $A\in\lambda$, then $S$ is
closed in $X$. Hence, it follows that $f$ is a quotient map.
\end{proof}

 If we let $\bigoplus\lambda$
denote the disjoint topological (direct) sum of the spaces $S\in
\lambda$, the inclusion maps $S\subseteq X$ combine to yield a
natural map $f:\,\bigoplus\lambda\to X$.

For natural covers the same result was proved by Franklin in
\cite{fran}: $X$ is a $\lambda$-space if and only if the natural
mapping $f:\,\bigoplus\lambda\to X$ is a quotient mapping.

\begin{corollary}{\it A Hausdorff space $X$ is a $\lambda$-space if and only
if for every $A\subseteq X$, the set $A$ is open in $X$ provided
that $A\cap \overline{Z}$ is open in $\overline{Z}$ for each $Z\in
\lambda$}.
\end{corollary}

\begin{proposition} Let $X$ be a topological space, $\lambda\subseteq \gamma\subseteq 2^X$. If $X$ is a $\lambda$-space then
$X$ is a $\gamma$-space.
\end{proposition}

\begin{proof} Let $S\subseteq X$ such that $S\cap \overline{A}$ is closed for each $A\in\gamma$. Since $X$ is a $\lambda$-space and $\lambda\subseteq \gamma$,
then $S$ is closed in $X$.
\end{proof}

\begin{theorem}\label{th116}  A continuous map $f:\,X\to Y$ of a topological space $X$
to a $\lambda$-space $Y$ is closed (open, quotient) if and only if
for every $A\in \lambda$ the restriction map
$f|_{f^{-1}(\overline{A})}:\, f^{-1}(\overline{A})\to\overline{A}$
is closed (open, quotient).
\end{theorem}

\begin{proof} Let $f:\,X\to Y$ be a closed (open) map, then the restriction map $f|_{f^{-1}(L)}$ of $f$ to $f^{-1}(L)$ is closed (open)
for every $L\subseteq Y$ (and, hence, for every $A\in\lambda$).
Let $f:\,X\to Y$ be a quotient mapping, then the restriction
$f|_{f^{-1}(L)}$ of $f$ to $f^{-1}(L)$ is quotient for every
closed set $L\subseteq Y$, and, hence, for the closure of elements
of $\lambda$.

Suppose that $f|_{f^{-1}(\overline{A})}:\,
f^{-1}(\overline{A})\to\overline{A}$ is closed (open) for each
$A\in\lambda$. Let $S$ be a closed (an open) set in $X$. Since
$f|_{f^{-1}(\overline{A})}$ is closed (open) and $f(S)\cap
\overline{A}=f(S\cap
f^{-1}(\overline{A}))=f|_{f^{-1}(\overline{A})}(S\cap
f^{-1}(\overline{A}))$ then $f(S)\cap \overline{A}$ is closed
(open) in $\overline{A}$ for each $A\in\lambda$. Since $Y$ is a
$\lambda$-space, then $f(S)$ is closed (open) in $Y$. It follows
that $f$ is closed (open).

Let $f|_{f^{-1}(\overline{A})}:\, f^{-1}(\overline{A})\to
\overline{A}$ be a quotient mapping for each $A\in\lambda$.
Consider $S\subseteq Y$ such that $f^{-1}(S)$ is closed in $X$.
The set $(f|_{f^{-1}(\overline{A})})^{-1}(S\cap
\overline{A})=f^{-1}(S)\cap f^{-1}(\overline{A})$ is closed in
$f^{-1}(\overline{A})$, hence, $S\cap\overline{A}$ is closed in
$\overline{A}$ for each $A\in\lambda$. Since $X$ is a locally
$\lambda$-space, the set $S$ is closed in $Y$. It follows
that the map $f$ is quotient.
\end{proof}

By Theorem \ref{th116}, we have the following result.

\begin{proposition} Let $X$ and $Y$ be Hausdorff spaces, $\lambda\subseteq 2^X$, $\gamma\subseteq 2^Y$, $f:\, X\to Y$
a quotient mapping, $X$ a $\lambda$-space, and $\lambda\subseteq
\{f^{-1}(\overline{A}):\, A\in\gamma\}$. Then $Y$ is a
$\gamma$-space.
\end{proposition}

\begin{definition} Let $X$ be a topological space, $\lambda\subseteq 2^X$. The space $X$ is called a {\it $\lambda_f$-space}
if whenever $f:\, X\to Y$ from $X$ into {\it any} Tychonoff space
$Y$ is continuous if and only if  $f|_{\overline{A}}:\,
\overline{A}\to Y$ is continuous for all $A\in\lambda$.
\end{definition}

In the case when $\lambda$ is a natural cover of $X$ the
following the notion of $\lambda_R$-space ($\Sigma_R$-space) was
introduced by Karnik and Willard in \cite{kawi}.

\begin{definition} Let $X$ be a topological space, $\lambda\subseteq 2^X$. The space $X$ is called a {\it $\lambda_R$-space}, if whenever $f$ is a real-valued function on
$X$ whose restriction to each $S\in \lambda$ is continuous (i.e.
whenever $f$ is a $\lambda$-continuous function on $X$), then $f$
is continuous on $X$.
\end{definition}

\begin{proposition} Let $X$ be a topological space, $\lambda\subseteq 2^X$. The space $X$ is a $\lambda_f$-space if and only if it is a
$\lambda_R$-space.
\end{proposition}

\begin{proof} Suppose that $X$ is a
$\lambda_R$-space and $f$ is a mapping from $X$ into a Tychonoff
space $Y$ such that $f|_{\overline{A}}:\, \overline{A}\to Y$ is
continuous for all $A\in\lambda$. Consider a homeomorphic
embedding $\varphi$ of $Y$ into
$\mathbb{R}^\kappa=\prod\{\mathbb{R}_{\alpha}: \alpha<\kappa\}$
for some $\kappa$ (since $Y$ is Tychonoff it can be embedded
in some Tychonoff cube of weight $\kappa$). Then
$\pi_{\alpha}\circ\varphi\circ f|_{\overline{A}}: A\rightarrow
\mathbb{R}_{\alpha}$ is continuous for each $\alpha<\kappa$ and
each $A\in \lambda$. Fix $\alpha<\kappa$. Since $X$ is a
$\lambda_R$-space, $\pi_{\alpha}\circ\varphi\circ f: X\rightarrow
\mathbb{R}_{\alpha}$ is continuous. It follows that $f$ is
continuous.
\end{proof}

The notion $\lambda_f$-space is a natural generalization of the
notion $k_f$-space \cite{kell}.

\begin{theorem}\label{th119} Let $X$ be a topological space, $\lambda\subseteq 2^X$. If $X$ is a $\lambda$-space then $X$ is a
$\lambda_f$-space.
\end{theorem}

\begin{proof} Suppose that $f:\,X\to Y$ such that  $f|_{\overline{A}}:\, \overline{A}\to Y$ is continuous for each $A\in\lambda$.
Let $B$ be a closed subset of $Y$. Then $f^{-1}(B)\cap
\overline{A}=(f|_{\overline{A}})^{-1}(B)$ is closed for each
$A\in\lambda$. Since $X$ is a $\lambda$-space,  $f^{-1}(B)$ is
closed. It follows that the mapping $f$ is continuous.
\end{proof}

Recall that a continuous mapping $f$ from $X$ onto $Y$ is called
an {\it $\mathbb{R}$-quotient map} provided whenever $g:
Y\rightarrow \mathbb{R}$, $g$ is continuous if and only if $g\circ
f$ is continuous. It follows immediately that $f$ is
$\mathbb{R}$-quotient if and only if $f$ is continuous and
whenever $g:Y\rightarrow Z$ where $Z$ is a Tychonoff space, then
$g$ is continuous if and only if $g\circ f$ is continuous. It is
also obvious that every quotient map is an $\mathbb{R}$-quotient
map, and easily verified that composition of $\mathbb{R}$-quotient
maps yields an $\mathbb{R}$-quotient map \cite{kawi}.

\begin{theorem} Let $X$ be a topological space, $\lambda$ be a cover of $X$. Then
the following statements are equivalent:

(1) $X$ is a $\lambda_f$-space;

(2) there are a Tychonoff space $Y$ and $\gamma\subseteq 2^Y$ such
that $Y$ is a locally $\gamma$-space, $X$ is an
$\mathbb{R}$-quotient image of $Y$ with respect to an
$\mathbb{R}$-quotient map $f:\, Y\to X$ and
$\gamma\subseteq\{B\subseteq Y:\ \exists A\in\lambda\ \mbox{and}\
f(B)=\overline{A}\}$.
\end{theorem}

\begin{proof} Let $f:\, Y\to X$ be an $\mathbb{R}$-quotient mapping and $Y$ be a locally $\gamma$-space. Let $g:\,
X\to \mathbb{R}$ is a mapping such that
$g|_{\overline{A}}:\, \overline{A}\to \mathbb{R}$ is continuous for any $A\in\lambda$.
Then for any point $y\in Y$ there is a neighborhood $U(y)$ of $y$
such that $\overline{U(y)}=\overline{B}$ for some $B\in \gamma$.
Then $g \circ f|_{B}$ is continuous as a composition of continuous
maps, and, hence, the mapping $g\circ f$ is continuous. Since the
mapping $f$ is an $\mathbb{R}$-quotient, $g$ is continuous.

Let $X$ be a Tychonoff $\lambda_f$-space and
$\bigcup\{A:\,A\in\lambda\}=X$. The space
$\tilde{X}=\bigoplus\{\overline{A}:\,A\in\lambda\}$ is a Tychonoff
locally $\gamma$-space where $\gamma$ is a family of members of
direct sum. Consider the family $\{i_A:\,A\in\lambda\}$ of maps
where $i_A:\, \overline{A}\to X$ is an identity map. Define the
mapping $f:\,\tilde{X}\to X$ as  $f|_{\overline{A}}=i_A$ for any
$A\in\lambda$. Since for any $A_1,\,A_2\in\lambda$\quad
$i_{A_2}|_{\overline{A_1}\cap\overline{A_2}}=i_{A_1}|_{\overline{A_1}\cap\overline{A_2}}$,
the mapping $f$ is well-defined and continuous.

Let $g:\,X\to \mathbb{R}$ be a $f$-continuous mapping. Clearly,
for any $A\in\lambda$ $g|_{\overline{A}}:\, \overline{A}\to
\mathbb{R}$ is continuous, hence, $g$ is continuous. Thus, we
proved that $f$ is $\mathbb{R}$-quotient.
\end{proof}

Thus, {\it $k_f$-spaces are Tychonoff spaces that can be represented as
$\mathbb{R}$-quotient spaces of Tychonoff locally compact spaces}.

\medskip
For natural covers the same result was proved by Karnik and
Willard in \cite{kawi}:  $X$ is a $\lambda_f$-space if and only if
$\phi:\,\bigoplus\lambda\to X$ is an $\mathbb{R}$-quotient map.

\begin{corollary}{\it Let $X$ and $Y$ be Tychonoff spaces, $\lambda\subseteq 2^X$, $\gamma\subseteq 2^Y$.
 The space $Y$ is an $\mathbb{R}$-quotient
image of the space $X$ with respect to an $\mathbb{R}$-quotient
mapping $f:X\to Y$, $\lambda\subseteq \{B\subseteq X:\ \exists
A\in\gamma\ \mbox{and}\ f(B)=\overline{A}\}$, and $X$ is a
$\lambda_f$-space. Then $Y$ is a $\lambda_f$-space.}
\end{corollary}

\begin{corollary}{\it If $f:\, X\to Y$ is a quotient mapping of a Tychonoff $k_f$-space $X$ onto a space
$Y$ then $Y$ is a Tychonoff $k_f$-space.}
\end{corollary}

\begin{proposition} Let $X$ be a topological space, $\lambda\subseteq\gamma\subseteq 2^X$. If $X$ is a
$\lambda_f$-space then $X$ is a $\gamma_f$-space.
\end{proposition}

\begin{proof} Let $f:\, X\to Y$ be a mapping from a space
$X$ to a Tychonoff space $Y$ such that
$f|_{\overline{A}}:\,\overline{A}\to Y$ is continuous for all
$A\in\gamma$. Then $f|_{\overline{B}}:\,\overline{B}\to Y$ is
continuous for all $B\in\lambda$, and, hence, the mapping $f$ is
continuous.
\end{proof}

For any topological space $X$ there exists the Tychonoff space
$X_{\tau}$ and the onto map $\tau_X: X\rightarrow X_{\tau}$ such
that for any map $f$ of $X$ to a Tychonoff space $Y$ there exists
a map $g: X_{\tau}\rightarrow Y$ such that $f=g\circ \tau_X$.
Evidently, for any map $f: X\rightarrow Y$ the unique map $\tau f:
X_{\tau}\rightarrow Y_{\tau}$ is defined such that $\tau f \circ
\tau_X=\tau_Y\circ f$. The correspondence $X\rightarrow X_{\tau}$
for any space $X$ and $f\rightarrow \tau f$ for any map $f$ is
called the {\it Tychonoff functor} \cite{Is,Ok}.

\begin{proposition} Let $X$ be a topological space, $\lambda\subseteq 2^X$.
If $X$ is a $\lambda$-space, then $X_{\tau}$ is a
$\lambda_f$-space.
\end{proposition}

\begin{proof} By definition of the Tychonoff functor, for a continuous real-valued function $f: X\rightarrow \mathbb{R}$ the
unique continuous real-valued function $\tau f:
X_{\tau}\rightarrow \mathbb{R}$ is defined such that $\tau f \circ
\tau_X=\mathbb{R}\circ f$. Thus, we can assume that
$C(X)=C(X_{\tau})$. It remains to apply Theorem \ref{th119}.
\end{proof}

Let us construct a modification of an arbitrary topological space
$X$ and an arbitrary family $\lambda\subseteq 2^X$ to a space
$X_{\lambda}$ consisting of the same set of points.

 Let $X$ be a topological space, $\lambda\subseteq 2^X$. Let
$Q=\{B\subseteq X:\, \overline{A}\cap B$ is closed in
$\overline{A}$ for all $A\in\lambda\}$. Let $X_{\lambda}$ be a set
$X$ with the topology $\tau=\{X\setminus B:\, B\in Q\}$. Let
$e:\,X_{\lambda}\to X$ be an identity map on $X$. Since the
topology on $X$ is weaker than the topology on $X_{\lambda}$, the
map $e$ is continuous. Thus $e$ is a condensation (= a continuous
bijection). Note that $e|_{\overline{A}}:\,\overline{A}_{\tau}\to
\overline{A}$ is a homeomorphism for each $A\in\lambda$. Also note
that $X_{\lambda}$ is a $\gamma$-space where
$\gamma=\{e^{-1}(A):\,A\in\lambda\}$. Further, we will call such
space as {\it $\lambda$-leader} of $X$.

In the case when $\lambda$ is a family of all non-empty compact
subsets of $X$, this modification was proposed by Cohen
\cite{coh}; for natural covers by Franklin \cite{fran}.

Further, we prove the uniqueness theorem for a $\lambda$-leader.

\begin{theorem}\label{th1118} Let $X$ be a topological space, $\lambda\subseteq 2^X$.
 Let $f_1:\,Y_1\to X$ and $f_2:\,Y_2\to X$ be condensations, $\gamma_1\subseteq 2^{Y_1}$ and $\gamma_2\subseteq 2^{Y_2}$ such that
$\{\overline{B}:\,B\in\gamma_1\}=\{f_1^{-1}(\overline{A}):\,A\in\lambda\}$
and
$\{\overline{C}:\,C\in\gamma_2\}=\{f_2^{-1}([A]):\,A\in\lambda\}$,
$f_1|_{[B]}:\,[B]\to X$ and $f_2|_{\overline{C}}:\,\overline{C}\to
X$ are homeomorphisms for any $B\in\gamma_1$ and $C\in\gamma_2$
and, let $Y_1$ be  a $\gamma_1$-space and $Y_2$ be a
$\gamma_2$-space. Then there is a homeomorphism $\phi:\,Y_1\to
Y_2$ such that $\{\phi(\overline{B}):\,
B\in\gamma_1\}=\{\overline{C}:\,C\in\gamma_2\}$.
\end{theorem}

\begin{proof} Let $\phi=f_2^{-1}\circ f_1$. Then
$\{\phi(\overline{B}):\,
B\in\gamma_1\}=\{\overline{C}:\,C\in\gamma_2\}$. It remains to
prove that $\phi$ is a homeomorphism. Since
$\phi_1|_{\overline{B}}:\,\overline{B}\to X$ and
$\phi_2|_{\overline{C}}:\,\overline{C}\to X$ is a homeomorphism
for any $B\in\gamma_1$ and $C\in\gamma_2$,
$\phi|_{\overline{B}}:\,\overline{B}\to Y_2$ is a homeomorphism.
Since $Y_1$ is a $\gamma_1$-space, $Y_2$ is a $\gamma_2$-space,
$\phi$ is a homeomorphism.
\end{proof}

In the case then $\lambda$ is a family  of all non-empty compact
subsets of $X$ Theorem \ref{th1118} was proved by Arhangel'skii in
\cite{arh}.

Note that $X_{\lambda}$ is always Hausdorff space for an arbitrary
Hausdorff space $X$. But, if $X$ is a Tychonoff space, then
$X_{\lambda}$ may not even be regular.

Recall that a topological space $X$ is called {\it functionally
Hausdorff} if for any two points $x, y \in X$ there is a
continuous real-valued function $f$ such that $f(x) \neq f(y)$.

\begin{proposition} Let $Z$ be a functionally Hausdorff non-regular
$k$-space and let $X$ be a Tychonoff modification of $Z$. Then a
$k$-leader of $X$ is non-regular.
\end{proposition}

Further, we want to propose a modification that does not deduce from
the class of Tychonoff spaces and leads to $\lambda_f$-spaces.

Let $X$ be a Tychonoff space, $\lambda\subseteq 2^X$. Let
$X_{\tau\lambda}$ be a Tychonoff modification of $\lambda$-leader
$X_{\lambda}$ of the space $X$. Note that the space
$X_{\tau\lambda}$ consists of the same set of points as $X$. In
addition, the topology on $X$ is weaker than the topology on
$X_{\tau\lambda}$, i.e. there is a continuous bijective map from
$X_{\tau\lambda}$ onto $X$. Let $e:\, X_{\tau\lambda}\to X$ be a
condensation of $X_{\tau\lambda}$ onto $X$. The topology on
$X_{\tau\lambda}$ is weaker than the topology on $X$, hence,
$e|_{e^{-1}(\overline{A})}:\, e^{-1}(\overline{A})\to
\overline{A}$ is homeomorphism for each $A\in\lambda$. Note also
that $X_{\tau\lambda}$ is a $\gamma_f$-space where
$\gamma=\{e^{-1}(\overline{A}):\, A\in\lambda\}$. Further, we will
call such space {\it $\lambda_f$-leader} of $X$ and denote by
$X_{\tau\lambda}$.

To conclude this section, we prove the uniqueness theorem for the
$\lambda_f$-leader.

\begin{theorem}\label{t1_1_19} Let $X$ be a Tychonoff space, $\lambda\subseteq 2^X$. Let
$f_1:\,Y_1\to X$ and $f_2:\,Y_2\to X$ be condensations, $\gamma_1$
and $\gamma_2$ be families of subsets of Tychonoff spaces $Y_1$
and $Y_2$, respectively, such that
$\{\overline{B}:\,B\in\gamma_1\}=\{f_1^{-1}(\overline{A}):\,A\in\lambda\}$
and
$\{\overline{C}:\,C\in\gamma_2\}=\{f_2^{-1}(\overline{A}):\,A\in\lambda\}$,
and also $f_1|_{\overline{B}}:\,\overline{B}\to X$ and
$f_2|_{\overline{C}}:\,\overline{C}\to X$ be homeomorphisms for
any $B\in\gamma_1$ and $C\in\gamma_2$ and, moreover, $Y_1$ be a
$\gamma_{1_f}$-space and $Y_2$ be a $\gamma_{2_f}$-space. Then
there is a homeomorphism $\phi:\,Y_1\to Y_2$ such that
$\{\phi(\overline{B}):\,
B\in\gamma_1\}=\{\overline{C}:\,C\in\gamma_2\}$.
\end{theorem}

\begin{proof} Let $\phi=f_2^{-1}\circ f_1$. Then
$\{\phi(\overline{B}):\,
B\in\gamma_1\}=\{\overline{C}:\,C\in\gamma_2\}$. We show that
$\phi$ is a homeomorphism. Clearly
$\phi|_{\overline{B}}:\,\overline{B}\to Y_2$ is a homeomorphism
for any $B\in\gamma_1$. Since $Y_1$ is a $\gamma_{1_f}$-space,
$\phi$ is continuous. The continuity of $\phi^{-1}$ is
checked similarly.
\end{proof}

\section{Embedding theorems for the topology of uniform convergence}

The idea of embedding one topological space into another is one of
the main ideas in topology and functional analysis. One of the
main embedding theorems is Tychonoff's theorem on the embedding of
a space of functions with the topology of pointwise convergence
into a Tychonoff product of uniform spaces \cite{tych}.

The first author of this paper has a wonderful idea for an
arbitrary Hausdorff space $X$ of constructing a Tychonoff space
$X'$ such that $C(X,Y)=C(X',Y)$ and $C_{\lambda,\mu}(X',Y)$ is
topological embedded in some "nice"\ space for any uniform space
$(Y,\mu)$.

The purpose of this section is to construct a model similar to the
Tychonoff's embedding, i.e., for any space $X$ and
$\lambda\subseteq 2^X$, embed the space $C_{\lambda,\mu}(X,Y)$ of
all continuous functions from $X$ into a uniform space $(Y,\mu)$
with the topology of uniform convergence on elements of $\lambda$
into some "nice"\ space.

In this section, by the neighborhood of a point, we mean the set
whose interior contains this point.

If $X$, $Y$ and $Z$ are spaces, define the composition function

$\Phi: C(X,Y)\times C(Y,Z)\rightarrow C(X,Z)$ by $\Phi(f,g)=g\circ
f$ for each $f\in C(X,Z)$ and $g\in C(Y,Z)$. Fixing one of the
components of the domain of the composition function results in
what is called an {\it induced function} (\cite{MckNta},
\cite{Arh89b}).

If $f\in C(X,Y)$, define the induced function $f^{\#}:C(Y,Z)\to
C(X,Z)$ by $f^{\#}(h)=h\circ f$ for every $h\in C(Y,Z)$. Induced
functions have been studied fairly well from a general topological
point of view. However, despite the fact that the "nice" \
topologies on $C(X,Y)$ are defined by uniformity, induced
functions have not been studied in detail as mappings of uniform
spaces.

\begin{theorem}\label{t1_2_3}
Let $X$ be a topological space, $\lambda\subseteq 2^X$,
$(Y,\mu)$ be a uniform space and let $f:\widetilde{X}\to X$ be a
condensation. Then $f^{\#}:C_{\lambda,\mu}(X,Y)\to
C_{f^{-1}(\lambda),\mu}(\widetilde{X},Y)$ is a uniform embedding for
$f^{-1}(\lambda)=\{f^{-1}(A):A\in\lambda\}$.
\end{theorem}

\begin{proof}
Let $\langle f^{-1}(A),M\rangle$ be an element of uniformity of
the space $C_{e^{-1}(\lambda),\mu}(\widetilde{X},Y)$, i.e.
$\langle f^{-1}(A),M\rangle =\{\langle f,g\rangle \in
C(\widetilde{X},Y)\times C(\widetilde{X},Y):\forall x\in
f^{-1}(A)\langle f(x),g(x)\rangle \in M\}$. Then for any $\langle
h_1,h_2\rangle\in\langle A,M\rangle$ and a point $x\in f^{-1}(A)$
the pair $\langle f^{\#}(h_1)(x),f^{\#}(h_2)(x)\rangle$ belong
$M$. Indeed $f^{\#}(h_1)(x)=h_1(f(x))$ and
$f^{\#}(h_2)(x)=h_2(f(x))$, since $x\in f^{-1}(A)$, $f(x)\in
A$. Hence, $\langle f^{\#}(h_1)(x),f^{\#}(h_2)(x)\rangle=\langle
h_1(f(x)),h_2(f(x))\rangle$. Thus, we proved that the pair
$\langle f^{\#}(h_1),f^{\#}(h_2)\rangle$ belongs $\langle
f^{-1}(A),M\rangle$, i.e. the mapping $f^{\#}$ is uniform
continuously. The uniform continuity of the inverse map is proven
similarly.

Let us check the injectivity of the mapping $f^{\#}$. Let $h_1,h_2\in
C_{\lambda,\mu}(X,Y)$ and $h_1\neq h_2$. Then there is a point
$x\in X$ such that $h_1(x)\neq h_2(x)$. Denote by
$\widetilde{x}=f^{-1}(x)$. Then
$f^{\#}(h_1(\widetilde{x}))=h_1(x)\neq
h_2(x)=f^{\#}(h_2(\widetilde{x}))$. Thus, it is proved
$f^{\#}(h_1)\neq f^{\#}(h_2)$, and, therefore, the theorem is
proves.
\end{proof}

\begin{corollary}\label{t1_2_4}
{\it Let $X$ be a Hausdorff space, $\lambda\subseteq 2^X$,
$(Y,\mu)$ be a Hausdorff uniform space. If $e:X_\lambda\to X$ is a
natural condensation of a $\lambda$-leader $X_\lambda$ onto $X$,
then the dual map $e^{\#}:C_{\lambda,\mu}(X,Y)\to
C_{e^{-1}(\lambda),\mu}(X_\lambda,Y)$ is uniform embedding.}
\end{corollary}

\begin{corollary}\label{t1_2_5}
{\it Let $X$ be a Tychonoff space, $\lambda$ be a cover of $X$,
$(Y, \mu)$ be a Hausdorff uniform space. If
\hbox{$e:X_{\tau\lambda}\to X$} is a natural condensation from a
$\lambda_f$-leader $X_{\tau\lambda}$ onto $X$ then the dual map
$e^{\#}:C_{\lambda,\mu}(X,Y)\to
C_{e^{-1}(\lambda),\mu}(X_{\tau\lambda},Y)$ is uniform embedding.}
\end{corollary}

Further, we are trying to establish how "nice" \ the space
$C_{e^{-1}(\lambda),\mu}(X_{\tau\lambda},Y)$.

Let $X$ be a Hausdorff $\lambda_f$-space where $\lambda$ is a
cover of $X$ and let $(Y,\mu)$ be a Hausdorff uniform space. Then
the space $C_{\lambda,\mu}(X,Y)$ is homeomorphic to a inverse
limit of the system $\{C_\mu(\overline{A},Y): A\in\lambda\}$.

 We order the family $\lambda$ as follows:
$A_2\leq A_1$ if, and only if, $\overline{A_2}\subseteq
\overline{A_1}$. By Remark~\ref{t1_2_2}, the family $\lambda$
directed by the relation $\leq$. For any pair $A_1,A_2\in\lambda$
such that $A_2\leq A_1$ defined mapping:
\index{$\pi^{A_1}_{A_2}:C_\mu(\overline{A_1},Y)\to
C_\mu(\overline{A_2},Y)$}
$\pi^{A_1}_{A_2}:C_\mu(\overline{A_1},Y)\to
C_\mu(\overline{A_2},Y)$, namely,
$\pi^{A_1}_{A_2}(f)=f\vert_{\overline{A_2}}$ for every $f\in
C_\mu(\overline{A_1},Y)$.

\begin{lemma}\label{t1_2_6}
Let $X$ be a Hausdorff space, $\lambda$ be a cover of $X$,
$(Y,\mu)$ be a uniform space. Then
$\pi^{A_1}_{A_2}:C_\mu(\overline{A_1},Y)\to
C_\mu(\overline{A_2},Y)$ is uniformly continuous for any
$A_1,A_2\in\lambda$ such that $A_2\leq A_1$.
\end{lemma}

\begin{proof}
Let $A_1,A_2\in\lambda$ such that $A_2\leq A_1$ and $\langle
\overline{A_2},M\rangle $ where $M\in\mu$, be an element of
uniformity on the space $C_\mu(\overline{A_2},Y)$. Then for any
functions $f_1,f_2\in C_\mu(\overline{A_1},Y)$ forming a pair
belonging $\langle\overline{A_1},M\rangle$, the pair $\langle
f_1\vert_{\overline{A_2}},f_2\vert_{\overline{A_2}}\rangle$
belongs $\langle\overline{A_2},M\rangle$, i.e., if $\langle
f_1,f_2\rangle \in\langle\overline{A_1},M\rangle$ then
$\langle\pi^{A_1}_{A_2}(f_1),\pi^{A_1}_{A_2}(f_2)\rangle\in\langle
\overline{A_2},M\rangle$. Hence, the mapping $\pi^{A_1}_{A_2}$ is
uniformly continuous.
\end{proof}

\begin{theorem}\label{t1_2_7}
Let $X$ be a Hausdorff $\lambda_f$-space, $\lambda$ be a cover of
$X$. If $(Y,\mu)$ is a Hausdorff uniform space, then
$C_{\lambda,\mu}(X,Y)$ is uniformly homeomorphic to the inverse
limit of the system
$S(X,\lambda,Y)=\{C_\mu(\overline{A},Y),\pi^{A_1}_{A_2},\lambda\}$.
\end{theorem}

\begin{proof}

For each $f=\{f_A\}\in\lim\limits_{\large\longleftarrow}
S(X,\lambda,Y)$ we define compatible mapping
$\bigtriangledown_{A\in\lambda}f_A:X\to Y$ as follows \cite{Eng}:
if $x\in A$ then $\bigtriangledown_{A\in\lambda}f_A(x)=f_A(x)$.
Since $\lambda$ is a cover of $X$ the mapping
$\bigtriangledown_{A\in\lambda}f_A$ is defined at each point of
$X$. Since for every $f=\{f_A\}$ the mappings $f_A$
($A\in\lambda$) are compatible (i.e.,
$f_{A_1}\vert_{\overline{A_1}\cap \overline{A_2}}=
f_{A_2}\vert_{\overline{A_1}\cap\overline{A_2}}$ for every pair
$A_1,A_2\in\lambda$), then $\bigtriangledown_{A\in\lambda}f_A$ is
 well-defined.

Define compatible mapping  $F:\lim\limits_{\large\longleftarrow}
S(X,\lambda,Y)\to C_{\lambda,\mu}(X,Y)$ as follows:
$F(f)=\bigtriangledown_{A\in\lambda}f_A$. The space $Y$ is
Tychonoff (Corollary 8.1.13 in \cite{Eng}), the map
$\bigtriangledown_{A\in\lambda}f_A$ is continuous on
$\overline{A}$ for every $A\in\lambda$ and $X$ is a
$\lambda_f$-space. Hence, the map
$\bigtriangledown_{A\in\lambda}f_A$ is continuous for every
$f\in\lim\limits_{\large\longleftarrow} S(X,\lambda,Y)$. It
follows that $F(\bigtriangledown_{A\in\lambda}f_A)\in
C_{\lambda,\mu}(X,Y)$ and, hence, $F$ is bijection.

Consider an arbitrary element of uniformity of the space
$\lim\limits_{\large\longleftarrow} S(X,\lambda,Y)$: $V=\{\langle
f,g\rangle: \langle f_{A_i},g_{A_i}\rangle \in\langle
\overline{A_i},M\rangle $ for $i=1,\ldots ,n\}$. Then, there is an
element $V'=\langle\bigcup\limits_{i=1}^n
\overline{A_i},\bigcap\limits_{i=1}^n M_i\rangle$ of uniformity of
the space $C_{\lambda,\mu}(X,Y)$ such that for any pair $\langle
h_1,h_2\rangle \in V'$ the pair $\langle
f^{-1}(h_1),f^{-1}(h_2)\rangle \in V$. Thus, $f^{-1}$ is uniform
continuous.

 Conversely, let $V=\langle A,M\rangle$ be an element
of uniformity of the space $C_{\lambda,\mu}(X,Y)$. Then there is
an element $V'=\{\langle f,g\rangle :\langle f_A,g_A\rangle
\in\langle \overline{A},M\rangle \}$ of uniformity of the space
$\lim\limits_{\large\longleftarrow} S(X,\lambda,Y)$ such that for
any pair $\langle h_1,h_2\rangle \in V'$ the pair $\langle
F(h_1),F(h_2)\rangle$ belongs to $V$.
\end{proof}

By Theorem~\ref{t1_2_7} and Corollary~\ref{t1_2_4} we have the
main theorem.

\begin{theorem}\label{t1_2_8}
Let $X$ be a Hausdorff topological space, $\lambda$ be a cover of
$X$. If $(Y,\mu)$ is a uniform space, then $C_{\lambda,\mu}(X,Y)$
is uniformly homeomorphic to subspace of the inverse limit
$\lim\limits_{\large\longleftarrow} S(X,\lambda,Y)$.
\end{theorem}

As corollaries to Theorems ~\ref{t1_2_7} and \ref{t1_2_8}, the
following propositions can be obtained, which were proved earlier
(see. \cite{Bur75b} and \cite{BecNarSuf}), but by other methods.

\begin{proposition}\label{t1_2_9}
Let $X$ be a Hausdorff $\lambda_f$-space, $\lambda$ be a cover of
$X$. If $Y$ is a Hausdorff uniform space with a complete
uniformity $\mu$ then $C_{\lambda,\mu}(X,Y)$ is a complete uniform
space.
\end{proposition}

\begin{proof}
Note that $C_\mu(\overline{A},Y)$ is complete (see Theorem 10,
\cite{kell}), and that a inverse limit of complete uniform spaces
is complete \cite{Eng}. Remain apply Theorem~\ref{t1_2_7}.
\end{proof}

\begin{proposition}\label{t1_2_9}
Let $X$ be a Hausdorff space, $\lambda$ be a cover of $X$ and
$(Y,\mu)$ be a uniform space. Then $C_{\lambda,\mu}(X,Y)$ is
complete if and only if $X$ is a $\lambda_f$-space.
\end{proposition}

 Proposition ~\ref{t1_2_9} was first proved in \cite{BecNarSuf}
where $\lambda$ is a family of all compact subsets of $X$ and
$Y=\mathbb{R}$. The proof of the general case does not differ from
the proof given in \cite{BecNarSuf}.

\begin{proposition}\label{t1_2_10}
Let $X$ be a Hausdorff space, $\lambda$ be a cover of $X$ and
$\{A_n\in\lambda: A_n\subseteq A_{n+1}, n\in \mathbb{N}\}$ be a
sequence of sets such that for any $B\in\lambda$ there is  $A_n$
such that $B\subseteq A_n$. If $Y$ is a metric space, then
$C_{\lambda,\mu}(X,Y)$ is metrizable.
\end{proposition}

\begin{proof}
Note that $C_\mu(\overline{A},Y)$ is metrizable by the metric

$d(h_1,h_2)=\min\{\sup\{\rho(h_1(x),h_2(x)):x\in\overline{A}\},1\}$
where $\rho$ is a metric of $Y$. Then the space
$\lim\limits_{\large\longleftarrow}\{C_\mu(\overline{A_n},Y),\pi^{A_1}_{A_2},
\mathbb{N}\}$ is metrizable, too (Corollary 4.2.5 \cite{Eng}).
But, by Proposition 1.2.5 (3) in \cite{FedFil},
$\lim\limits_{\large\longleftarrow} S(X,\lambda,Y)$ is
homeomorphic to
$\lim\limits_{\large\longleftarrow}\{C_\mu(\overline{A_n},Y),\pi^{A_1}_{A_2},\mathbb{N}\}$
and, hence, is metrizable. By Theorem \ref{t1_2_8},
$C_{\lambda,\mu}(X,Y)$ is metrizable.
\end{proof}

The following theorem gives additional embedding properties in
the case when $X$ is a Tychonoff space.

\begin{theorem}\label{t1_2_11}
Let $X$ be a Tychonoff space, $\lambda$ be a cover of $X$ and
$(Y,\mu)$ be a Hausdorff uniform space. If $e:X_{\tau\lambda}\to
X$ is a natural condensation of $X_{\tau\lambda}$ onto $X$, then
$C_{e^{-1}(\lambda),\mu}(X_{\tau\lambda},Y)$ is the completion of
$C_{\lambda,\mu}(X,Y)$.
\end{theorem}

\begin{proof} By Proposition~\ref{t1_2_9}, the space
$C_{e^{-1}(\lambda),\mu}(X_{\tau\lambda},Y)$ is complete. Then the
completion of $C_{\lambda,\mu}(X,Y)$ is homeomorphic to some
$Z\subseteq C_{e^{-1}(\lambda),\mu}(X_{\tau\lambda},Y)$.
Obviously, $Z$ is equal to the closure of $C_{\lambda,\mu}(X,Y)$
in the space $C_{e^{-1}(\lambda),\mu}(X_{\tau\lambda},Y)$.
Consider a topology on the set $X$ generated by a base consisting
of the sets $f^{-1}(U)$, where $f\in Z$, $U$ is open in $Y$. We
denote the resulting space as $\widetilde{X}$. Then
$Z=C_{\lambda,\mu}(\widetilde{X},Y)$. Since the space $Z$ is
complete, by Proposition~\ref{t1_2_9}, the space
$\widetilde{X}$ is a $\lambda_f$-space. The topology on the space
$\widetilde{X}$ is stronger than a topology on $X$ and is weaker
than a topology on $X_{\tau\lambda}$, i.e. the space
$X_{\tau\lambda}$ has a condensation onto $\widetilde{X}$, and
$\widetilde{X}$ has a condensation onto $X$. Since the spaces
$X_{\tau\lambda}$ and $X$ induce on elements $\lambda$ the same
topology, the space $\widetilde{X}$ induces on elements
$\lambda$ the same topology. Hence, $\widetilde{X}$ is a
$\lambda$-leader of $X$. Then, by Theorem~\ref{t1_1_19}, the space
$\widetilde{X}$ is homeomorphic to $X_{\tau\lambda}$. It follows
that the space $Z$ is equal to the space
$C_{\lambda,\mu}(\widetilde{X},Y)$ and the space
$C_{e^{-1}(\lambda),\mu}(X_{\tau\lambda},Y)$. Thus,
$C_{e^{-1}(\lambda),\mu}(X_{\tau\lambda},Y)$ is the completion of
$C_{\lambda,\mu}(X,Y)$.
\end{proof}

\section{Embedding theorems for the (weak) set-open topology}

For a space $C_\lambda(X,Y)$ we get the following result.

\begin{theorem}\label{t1_2_13}
Let $X$ and $Y$ be topological spaces, $\lambda\subseteq 2^X$ and
$e:X_\lambda\to X$ be a natural condensation of $\lambda$-leader
$X_\lambda$ onto $X$ then $e^{\#}:C_\lambda(X,Y)\to
C_{e^{-1}(\lambda)}(X_\lambda,Y)$ is embedding.
\end{theorem}

\begin{proof}
Let $[A,V]=\{f\in C_\lambda(X,Y):f(A)\subseteq V\}$ be an element
of a subbase of $C_\lambda(X,Y)$. Consider $[e^{-1}(A),V]= \{g\in
C_{e^{-1}(\lambda)}(X_\lambda,Y):g(e^{-1}(A))\subseteq V\}.$
Obviously, it belongs to a subbase of the space
$C_{e^{-1}(\lambda)}(X_{\lambda},Y)$. We show that
$e^{\#}([A,V])=e^{\#}(C_\lambda(X,Y))\bigcap [e^{-1}(A),V]$.
Indeed, if $f\in [A,V]$, then
$e^{\#}(f)(e^{-1}(A))=f(e(e^{-1}(A)))=f(A)\subseteq V$, i.e.
$e^{\#}(f)\in [e^{-1}(A),V]$. Conversely, let $g\in
[e^{-1}(A),V]\bigcap e^{\#}(C_\lambda(X,Y))$ then there is $g'$
such that $e^{\#}(g')=g$ and $g'(A)=g(e^{-1}(A))\subseteq V$.
Thus, we show that the maps $e^{\#}$ and $(e^{\#})^{-1}$ are
continuous. The one-to-one mapping $e^{\#}$ is proved in the same
way as in Theorem ~\ref{t1_2_5}.
\end{proof}

Let $A$ be a subset of $X$ and $\lambda\subseteq 2^X$. Denote by
$\lambda_A=\{A\bigcap B:B\in \lambda\}$.

\begin{theorem}\label{t1_2_14}
Let $X$ and $Y$ be Hausdorff topological spaces, $\lambda$ be a
cover of $X$ such that for any $A_1,A_2\in\lambda$, the sets
$A_1\bigcup A_2$ and $A_1\bigcap A_2$ belongs to $\lambda$. If $X$
is a $\lambda_f$-space then $C_\lambda(X,Y)$ is homeomorphic to an
inverse limit of the system
$S_{*}(X,\lambda,Y)=\{C_{\lambda_{\overline{A}}}(\overline{A},Y),\pi^{A_1}_{A_2},\lambda\}$.
\end{theorem}

\begin{proof}
Let $F:S_{*}(X,\lambda,Y)\to C_\lambda(X,Y)$ be a compatible
mapping. The one-to-one mapping is proved in the same way as in
Theorem~\ref{t1_2_7}.

Let us claim that $F$ is a homeomorphism. A subbase of the space
$S_{*}(X,\lambda,Y)$ consists of sets of the form
$\pi^{-1}_{\overline{A}}([B,V]_{\overline{A}})$ where
$\pi_{\overline{A}}$ is a projection of $S_{*}(X,\lambda,Y)$ on
$C_{\lambda_{\overline{A}}}(\overline{A},Y)$, and
$[B,V]_{\overline{A}}$ is an element of subbase of the space
$C_{\lambda_{\overline{A}}}(\overline{A},Y)$. We prove that, the
mapping transforms this subbase into a standard subbase on
$C_\lambda(X,Y)$ consisting of sets of the form $[A,V]$.

Indeed, $F(\pi^{-1}_{\overline{A}}([B,V]_{\overline{A}}))=[B,V]$
and $f^{-1}([A,V])=\pi^{-1}_{\overline{A}}([A,V]_{\overline{A}})$.
Hence, $F$ is a homeomorphism.
\end{proof}

\begin{corollary}\label{t1_2_15}
{\it Let $X$ and $Y$ be a Hausdorff topological spaces, $\lambda$
be a cover of $X$ and $\lambda$ be closed with respect to finite
union and finite intersection. Then $C_\lambda(X,Y)$ is embedding
into an inverse limit of $S_{*}(X,\lambda,Y)$.}
\end{corollary}

For a space $C_{\lambda^*}(X,Y)$ we get the similar result.

\begin{theorem}\label{t1_2_131}
Let $X$ and $Y$ be topological spaces, $\lambda\subseteq 2^X$ and
$e:X_\lambda\to X$ be a natural condensation of $\lambda$-leader
$X_\lambda$ onto $X$ then $e^{\#}:C_{\lambda^*}(X,Y)\to
C_{e^{-1}(\lambda)^*}(X_\lambda,Y)$ is embedding.
\end{theorem}

\begin{theorem}\label{t1_2_142}
Let $X$ and $Y$ be Hausdorff topological spaces, $\lambda$ be a
cover of $X$ such that for any $A_1,A_2\in\lambda$, the sets
$A_1\bigcup A_2$ and $A_1\bigcap A_2$ belongs to $\lambda$. If $X$
is a $\lambda_f$-space, then $C_{\lambda^*}(X,Y)$ is homeomorphic
to an inverse limit of the system
$S_{*}(X,\lambda,Y)=\{C_{\lambda^*_{\overline{A}}}(\overline{A},Y),\pi^{A_1}_{A_2},\lambda\}$.
\end{theorem}

\section{Algebraic properties of function spaces}

In papers \cite{os4,os6,os11}, the algebraic properties of the
space $C_{\lambda}(X,Y)$ was investigated. In this section we
study the following question: under what conditions should the
 construction, constructing in the previous section, preserves the algebraic
properties which naturally arise on $C_{\lambda,\mu}(X,Y)$.

\subsection{Topological groups}

 Let $Y$ be a Hausdorff abelian group with the identity element $e$ and $X$ be a Hausdorff space. Then $C(X,Y)$ is an
abelian group respect of operation of the pointwise
multiplication, i.e. $(f*g)(x)=f(x)*g(x)$ for every $x\in X$.
Group operation in $Y$ we denote by $*$ and an inverse element for
$b$ we denote by $b^-$.

Consider the function $i\in C(X,Y)$ such that $i(x)=e$ for every
$x\in X$. Then $i$ is an identity element of the group $C(X,Y)$.
On the set $C(X,Y)$, one can introduce a topology consistent with
the group structure, specifying a base neighborhoods of $i$ as
follows.

Let $\lambda$ be a cover of $X$, $\beta_c=\{[A,V]: A\in\lambda,
V\in\beta\}$ where $\beta$ is a fundamental family of
neighborhoods of $e$ in $Y$ and $[A,V]=\{f\in C(X,Y):f(A)\subseteq
V\}$.

Proposition \ref{t1_3_1} shows that the system $\beta_c$ satisfies
the conditions necessary and sufficient for the shifts of the
elements $\beta_c$ to form the base of the topology consistent
with the group structure of $C(X,Y)$. The topology described above
will be denoted as $C_{\lambda,G}(X,Y)$.

\begin{proposition}\label{t1_3_1} For the family $\beta_c$ holds:
\begin{enumerate}
\item For every $W\in\beta_c$ there is $W^1\in\beta_c$ such that
$W^1*W^1\subseteq W$; \item For every $W\in\beta_c$,
$W^-\in\beta_c$,
\end{enumerate}

i.e. $C_{\lambda,G}(X,Y)$ is a topological group.
\end{proposition}

\begin{proof} Claim (1). Let
$W\in\beta_c$, then $W=[A,V]$ where $V$ is a neighborhood of $e$
in $Y$. For $V$ there is $V^{1}$ such that $V^{1}*V^{1}\subseteq
V$. The set $W^{1}=[A,V^{1}]$ belongs to $\beta_c$ and
$W^{1}*W^{1}$ = $\{f*g:f\in W^{1},g\in W^{1}\}$ =$\{f*g:f(A)\in
V^{1} \;$ and $\; g(A)\in V^{1}\}$. Clearly, $f*g$ mapping the set
$A$ in the set $V^{1}*V^{1}\subseteq V$, i.e.
$W^{1}*W^{1}\subseteq W$.

Claim (2). Let $W\in\beta_c$, then $W=[A,V]$. The set
$W^-=\{f^-:f\in W\}=\{f^-:f(A)\subseteq V\}=\{g(A):g(A)\subseteq
V^-\}$, i.e.,  $W^-=[A,V^-]$  and, hence, belongs to $\beta_c$. By
Proposition 1 (3) in \cite{Bur69}, $C_{\lambda,G}(X,Y)$ is a
topological group.
\end{proof}

If $Y$ is a Hausdorff abelian group, then  $Y$ has a uniform
structure consistent with the topology on $Y$ as follows: if
$\beta$ is a base of neighborhoods of $i$ in $Y$, then the system
$\{\mu_V:V\in\beta\}$ of neighborhoods of diagonal where
$\mu_V=\{<x,y>\in Y\times Y:y*x^-\in V\}$ is a base of uniformity
$\mu$ on $Y$. This uniformity, we will call {\it natural}.

\begin{proposition}\label{t1_3_2} Let $X$ be Hausdorff space, $\lambda$ be a cover of $X$, $Y$ be a Hausdorff abelian group and
$\mu$ be a natural uniformity on $Y$. Then the topologies of
spaces $C_{\lambda,G}(X,Y)$ and $C_{\lambda,\mu}(X,Y)$ coincide.
\end{proposition}

\begin{proof}
Let $f*U(A,V)$ be a neighborhood of a function $f\in
C_{\lambda,G}(X,Y)$. Then $f*U(A,V) =\{f*g:g\in
U(A,V)\}=\{f*g:g(A)\subseteq V\}$. Let $<A,M>$ be an element of
uniformity on  $C_{\lambda,\mu}(X,Y)$, $<A,M>=\{<f,g>:\forall x\in
A <f(x),g(x)> \in M\}$ where $M=\{<x,y>:y*x^-\in V\}$. A
neighborhood of $f$ with respect to this element of uniformity is
the set
$$<A,M>(f)=\{g\in C(X,Y):<f,g>\in <A,M>\}=$$
$$=\{g\in C(X,Y):\forall x\in A\Rightarrow g(x)*f(x)^-\in V\}=$$
$$\{f*g:\forall x\in A\Rightarrow g(x)\in V\}=f*[A,M].$$
Thus, topologies $C_{\lambda,G}(X,Y)$ and $C_{\lambda,\mu}(X,Y)$
are coincide.

\end{proof}

\begin{remark}\label{t1_3_3}
A uniformity of the space $C_{\lambda,\mu}(X,Y)$ is a natural
uniformity of topological group $C_{\lambda,G}(X,Y)$.
\end{remark}

\begin{corollary}\label{t1_3_4}
{\it Let $X$ be a Hausdorff space, $\lambda$ be a cover of $X$,
$Y$ be a Hausdorff abelian group and $\mu$ be a natural uniformity
on $Y$. Then $C_{\lambda,\mu}(X,Y)$ is a Hausdorff abelian group.
In particular, for a Hausdorff space $X$ and a cover $\lambda$ of
$X$ the space $C_{\lambda,\rho}(X)$ is a topological group.}
\end{corollary}

\begin{proof} By Propositions \ref{t1_2_6} and \ref{t1_2_7},
$C_{\lambda,\mu}(X,Y)$ is a topological group. By results in
\cite{Bur75}, $C_{\lambda,\mu}(X,Y)$ is a Hausdorff abelian group.

In conclusion, note that the natural metric generates the natural
uniform structure of the additive group that completes the proof
of the corollary.
\end{proof}

\begin{theorem}\label{t1_3_5}
Let $X$ be a Hausdorff (Tychonoff) topological space, $\lambda$ be
a cover of $X$, $Y$ be a Hausdorff abelian group, and $\mu$ be a
natural uniformity on $Y$. If $e:X_\lambda\to X$
($e:X_{\tau\lambda}\to X$) is a natural condensation from
$\lambda$-leader $X_\lambda$ ($\lambda_f$-leader
$X_{\tau\lambda})$ onto $X$, then
$$e^\#:C_{\lambda,\mu}(X,Y)\to C_{e^{-1}(\lambda),\mu}(X_\lambda,Y)\;\;
(e^\#:C_{\lambda,\mu}(X,Y)\to
C_{e^{-1}(\lambda),\mu}(X_{\tau\lambda},Y))$$ is a group
isomorphism.
\end{theorem}

\begin{proof}
 By Corollary
\ref{t1_3_3}, $C_{\lambda,\mu}(X,Y)$,
$C_{e^{-1}(\lambda),\mu}(X_{\tau\lambda},Y)$ and
$C_{e^{-1}(\lambda),\mu}(X_\lambda,Y)$ are topological groups. By
Theorem \ref{t1_2_3}, the map $e^\#$ is uniform embedding. Remain
prove that the map $e^\#$ saves the operation. Indeed,
$e^\#(f*g)(x)=(f*g)(e(x))=f(e(x))*g(e(x))=e^\#(f)(x)*e^\#(g)(x)$,
i.e., the map $e^\#$ saves the operation.
\end{proof}

 By Theorem \ref{t1_3_5}, the group $C_{\lambda,\mu}(X,Y)$ we can consider as a subgroup of group
$C_{e^{-1}(\lambda),\mu}(X_{\tau\lambda},Y)$ or as a subgroup of
group $C_{e^{-1}(\lambda),\mu}(X_\lambda,Y)$.

\begin{theorem}\label{t1_3_6}
Let $X$ be a Hausdorff $\lambda_f$-space where $\lambda$ be a
cover of $X$.  If $Y$ is a Hausdorff abelian group and $\mu$
is a natural uniformity of $Y$, then $C_{\lambda,\mu}(X,Y)$ is
isomorphic to the inverse limit $S(X,\lambda,Y)$ of topological
groups $C_\mu(\overline{A},Y)$.
\end{theorem}

\begin{proof}
Let us prove that the map
$\pi^{A_1}_{A_2}$ saves the group operation. Indeed, let
$A_1,A_2\in\lambda$ and $A_1\ge A_2$, $f_1,f_2\in
C_\mu(\overline{A},Y)$, then
$$\pi^{A_1}_{A_2}(f_1*f_2)=(f_1*f_2)\vert_{\overline{A_2}}=
f_1\vert_{\overline{A_2}}*f_2\vert_{\overline{A_2}}=
\pi^{A_1}_{A_2}(f_1)*\pi^{A_1}_{A_2}(f_2).$$ Thus, we proved that
$S(X,\lambda,Y)$ is a inverse limits of topological groups. Remain
to check that the mapping $F: \lim\limits_\leftarrow
S(X,\lambda,Y)\to C_{\lambda,\mu}(X,Y)$, constructing in the proof
of Theorem \ref{t1_2_7}, preserved the operation $*$. Let
$f_1,f_2\in\lim\limits_\leftarrow S(X,\lambda,Y)$, i.e.
$f_1=\{f^1_A\}_{A\in\lambda}$, $f_2=\{f^2_A\}_{A\in\lambda}$.
Clearly $f_1*f_2=\{f^{1}_A*f^{2}_A\}_{A\in\lambda}$. But, then
$\nabla_{A\in\lambda}f^{1}_A(x)*\nabla_{A\in\lambda}f^{2}_A(x)=
\nabla_{A\in\lambda}(f^{1}_A*f^{2}_A)(x)$.
\end{proof}

The following results are corollaries of Theorems \ref{t1_3_4},
\ref{t1_2_11} and \ref{t1_3_6}.

\begin{corollary}\label{t1_3_7}
{\it Let $X$ be a Hausdorff space, $\lambda\subseteq 2^X$ and $Y$
be a Hausdorff abelian group. If $\mu$ is a natural
uniformity on $Y$ then $C_{\lambda,\mu}(X,Y)$ is a subgroup of the
inverse limit system of topological groups
$C_\mu(\overline{A},Y)$ where $A\in\lambda$. In particular,
$C_{\lambda,\rho}(X,Y)$ is a subgroup of the inverse limit system
of metrizable groups $C_{\rho}(\overline{A})$ where
$A\in\lambda$.}
\end{corollary}

\begin{corollary}\label{t1_3_8}
{\it Let $X$ be a Tychonoff space, $\lambda\subseteq 2^X$ and $Y$
be a complete abelian topological group. If $\mu$ is a
natural uniformity on $Y$, then the inverse limit system of
topological groups $\lim\limits_\leftarrow S(X,\lambda,Y)$ is the
completion of the space $C_{\lambda,\mu}(X,Y)$. In particular,
$C_{\lambda,\rho}(X,Y)$ is a dense subgroup of the inverse limit
of topological groups $C_{\rho}(\overline{A})$ where
$A\in\lambda$.}
\end{corollary}

 Thus, we can consider the topological group
$C_{\lambda,\mu}(X,Y)$ as a subgroup (as a dense subgroup in the
case then $X$ is Tychonoff) of the inverse limit of topological
groups $C_{\rho}(\overline{A},Y)$.

\subsection{Topological rings}

  In (\cite{Bur69} 6 , 111) Bourbaki note that
 if $Y$ is a topological ring then $C(X,Y)$ is a ring where the
 ring operations are pointwise addition and multiplication.

A natural question arises: for which families $\lambda$ will the
space $C_{\lambda,\mu}(X,Y)$ be a topological ring ($\mu$ is a
natural uniformity of additive group on $Y)$ ?

Theorem \ref{t1_3_9} is an answer on this question.

\medskip

Recall that a set $B\subseteq Y$ is called {\it bounded} (in
topological ring), if for any neighborhood $V$ of the identity
element there is a neighborhood $W$ of the identity element that
$W*B\subseteq V$ and $B*W\subseteq V$.

\begin{theorem}\label{t1_3_9} Let $Y$ be a Hausdorff topological ring,
$\mu$ be a natural uniformity of an additive group $Y$, $X$ be a
Hausdorff topological space and $\lambda$ be a family of subsets
of $X$. Then $C_{\lambda,\mu}(X,Y)$ is a topological ring if and
only if for any $A\in \lambda$ and any continuous function
$f:X\to Y$ the set $f(A)$ is bounded in $Y$.
\end{theorem}

\begin{proof} Let us prove that $\beta_c$ is a filter of neighborhoods of zero of $C_{\lambda,\mu}(X,Y)$. The following conditions
(1)--(4):

\begin{enumerate}
\item For every $W\in\beta_c$ there is $W^1\in\beta_c$ such that
$W^{1}+W^{1}\subseteq W$; \item For every $W\in\beta_c$,
$-W\in\beta_c$; \item For every $f\in C(X,Y)$ and $W\in\beta_c$,
there is  $W^1\in\beta_c$ such that $f*W^{1}\subseteq W$ and
$W^{1}*f\subseteq W$; \item For every $W\in\beta_c$ there is
$W^1\in\beta_c$ such that $W^{1}*W^{1}\subseteq W$;
\end{enumerate}

holds if, and only if, $A$ is $Y$-bounded (i.e., for any $f\in
C(X,Y)$ the set $f(A)$ is bounded in $Y$) for any $A\in \lambda$
\cite{Bur69}.

By Proposition \ref{t1_3_3}, a filter of neighborhoods of zero in
$C_{\lambda,\mu}(X,Y)$ consists of sets of the form $[A,V]=\{f\in
C(X,Y):f(A)\subseteq V\}$, where $A\in\lambda$, and $V\in\beta$,
$\beta$ is a filter neighborhoods of zero in $Y$. By Proposition
\ref{t1_3_1}, the conditions (1) and (2) holds, for any family
$\lambda$.

Let us prove the condition (4). Let $W=[A,V] \in \beta_c$ where
$A\in\lambda$ and $V\in\beta$. Then there is $V^1\in\beta$ such
that $V^{1}*V^{1}\in V$. We show that for $W^{1}=U(A,V^{1})$ holds
$W^{1}*W^{1}\in W$. Indeed, $W^{1}*W^{1}=\{f*g:f\in W^{1},g \in
W^{1}\}= \{f*g:f(A)\subseteq V^{1}$ and $g(A)\subseteq V^{1}\}$.
Clearly, for every point $x\in A$ $f(x)*g(x)\subseteq
V^{1}*V^{1}\in V$. Hence, $(f*g)(A)\subseteq V$ and
$W^{1}*W^{1}\in W$. Note that the condition (4), as well as
conditions (1) and (2), are true for any family $\lambda$.

Let $\lambda$ be a family of $Y$-bounded subsets of $X$,
$W=[A,V]\in\beta_c$ and $f\in C(X,Y)$. Then there is the
neighborhood $V^{1}$ of zero in $Y$ such that $f(A)*V^{1}\in V$
and $V^{1}*f(A)\in V$. We show that the condition (3) holds for
$W^{1}=[A,V^{1}]$. It is enough to show that any function $g\in
f*W^{1}$ and any function $g^{1}\in W^{1}*f$ mapping the set $A$
in $V$. Indeed, let $g=f*h$ and $g^{1}=h^{1}*f$ where $h,
h^{1}\in W^{1}$. Then, for any point $x\in A g(x)=f(x)*h(x)\in
f(A)*V^{1}$ and $g^{1}(x)=h^{1}(x)*f(x)\in V^{1}*f(A)$. It follows
that $g(A)\in V$ and $g^{1}(A)\in V$.

Let condition (3) be satisfied, we prove that $f(A)$ is bounded in
$Y$ for any $A\in \lambda$ and $f\in C(X,Y)$. Let $A\in\lambda$
and $W=[A,V]$ be a neighborhood of zero in $C_{\lambda,\mu}(X,Y)$,
and $f\in C(X,Y)$. Then, by condition (3), there is a neighborhood
$W^{1}=[A^{1},V^{1}]$ of zero in $C_{\lambda,\mu }(X,Y)$ such that
$f*W^{1} \in W$ and $W^{1}*f \in W$ (by Remark \ref{t1_2_2}, we
can assume that $A^{1}=\overline{A^{1}})$. The set $H=\{h \in
C(X,Y):\forall x\in X$ $h(x)=y \in V^{1}\}$ is the set of all
constant functions from $W^{1}$. Then $f*H$ and $H*f$ are subsets
of $W$, i.e. for any $h\in H$ and $x\in A^{1}$ $(f*h)(x) \in V$
and $(h*f)(x) \in V$. Clearly $A\subseteq\overline{A^{1}}=A^{1}$,
hence, $f(A)*V^{1}\subseteq V$ and $V^{1}*f(A)\subseteq V$. Thus,
we get that $f(A)$ is bounded.
\end{proof}

By Theorem \ref{t1_3_5} we have the following result.

\begin{corollary}\label{t1_3_10}{\it Let $X$ be a topological space. A
space $C_{\rho}(X)$ is a topological ring if and only if $X$ is
pseudocompact.}
\end{corollary}

Since a relatively compact subset of a topological ring is bounded
(see \cite{Bur69} 3, \S 6, ex.12), we have the following results.

\begin{corollary}\label{t1_3_11} {\it Let $X$ be a Hausdorff space, $\lambda$ be a family of bounded subsets of $X$. Then
$C_{\lambda,\rho}(X)$ is a topological ring.}
\end{corollary}

\begin{proof} Note that an continuous image of a topological
bounded set is a bounded set, and a bounded set in $\mathbb{R}$ is
relatively compact.
\end{proof}

\begin{corollary}\label{t1_3_12} {\it Let $X$ be a Hausdorff space, $Y$ be a
Hausdorff topological ring. Then $C_c(X,Y)$ and $C_p(X,Y)$ are
topological rings.}
\end{corollary}

The following theorem continues to clarify the properties of the
embedding of function spaces.

\begin{theorem}\label{t1_3_13} Let $X$ be a Hausdorff (Tychonoff)
space, $Y$ be a Hausdorff topological ring and $\mu$ be a natural
uniformity of additive group $Y$. Let $\lambda$ be a family of
subsets of $X$ such that if $e:X_{\lambda} \rightarrow X$
($e:X_{\tau\lambda} \rightarrow X)$ is a natural condensation from
the $\lambda$-leader $X_\lambda$ ($\lambda_f$-leader
$X_{\tau\lambda}$) onto $X$, and every continuous image of the set
$e^{-1}(A)$ is bounded in $Y$ for any $A\in\lambda$. Then $e^\#:
C_{\lambda,\mu}(X,Y) \to C_{e^{-1}(\lambda),\mu}(X_\lambda,Y)$
($e^\#:C_{\lambda,\mu}(X,Y) \to
C_{e^{-1}(\lambda),\mu}(X_{\tau\lambda},Y))$ is a ring
homeomorphism onto $e^\#(C_{\lambda,\mu}(X,Y))$.
\end{theorem}

\begin{proof} By Theorem
\ref{t1_2_10}, topologies of spaces
$C_{\lambda,\mu}(X,Y),\linebreak
C_{e^{-1}(\lambda),\mu}(X_\lambda,Y)$ and
$C_{e^{-1}(\lambda),\mu}(X_{\tau\lambda},Y)$ compatible with
ring-structured of these spaces. By Theorem \ref{t1_2_3},
$e^\#:C_{\lambda,\mu}(X,Y) \to
C_{e^{-1}(\lambda),\mu}(X_\lambda,Y)$ and
$e^\#:C_{\lambda,\mu}(X,Y) \to
C_{e^{-1}(\lambda),\mu}(X_{\tau\lambda},Y)$ are uniform embedding.
To complete the proof, it is enough to check the preservation of
operations. Indeed, for every $x\in X$
$e^\#(f$+$g)(x)=(f$+$g)(e(x))=
f(e(x))$+$g(e(x))=e^\#(f)(x)$+$e^\#(g)(x)$, i.e.
$e^\#(f$+$g)$=$e^\#(f)$+$e^\#(g)$. Preserving the operation $*$ is
checked in the same way.
\end{proof}

\begin{corollary}\label{t1_3_14} {\it Let $X$ be a Hausdorff space,
$X_{k}$ be a $k$-leader of $X$ and $Y$ be a topological ring. Then
$C_c(X,Y)$ is a subring of the ring $C_c(X_k,Y)$.}
\end{corollary}

\begin{corollary}\label{t1_3_15}{\it Let $X$ be a Hausdorff space, ${\it ps}$ be a
family of all pseudocompact subsets of $X$. Then $C_{ps,\rho }(X)$
is a subring of the ring $C_{e^{-1}(ps),\rho}(X_{ps})$.}
\end{corollary}

Denote by $C^{*}(X,Y)$ the set of all continuous functions from
$X$ into $Y$ such that $f(X)$ is bounded in $Y$. Note that if
$\mu$ is a natural uniformity on a topological ring $Y$ then
$C^{*}_\mu(X,Y)$ is a topological ring.

\begin{theorem}\label{t1_3_16} Let $X$ be a Hausdorff space, $\lambda$ be a cover of $X$ and $X$ be a $\lambda_f$-space. Let $Y$
be a Hausdorff topological ring. Moreover, for every $f\in C(X,Y)$
and $A\in\lambda$ the set $f(A)$ is bounded in $Y$. If $\mu$ is a
natural uniformity on  $Y$ then $C_{\lambda,\mu}(X,Y)$ is
isomorphic to the inverse limit of the system $S^{*}(X,\lambda,Y)=
\{C^{*}_\mu(\overline{A},Y), \pi^{A_1}_{A_2},\lambda\}$.
\end{theorem}

\begin{proof} We need to check the following conditions: (1) maps $\pi^{A_1}_{A_2}$ are continuous and
preserved of operations of the ring $C_{\lambda,\mu}(X,Y)$;

 (2) mapping $F: \lim\limits_\leftarrow S^{*}(X,\lambda,Y)
\rightarrow C_{\lambda,\mu}(X,Y)$ (similar the mapping $F$ in the
proof of Theorem \ref{t1_2_7}) is a homeomorphism, preserving of
ring operations.

Let us check the condition (1). By Lemma \ref{t1_2_6}, $\pi^{A_1}_{A_2}$ are
uniform continuous.

Similar the proof of Theorem \ref{t1_3_6}, operations of
$\pi^{A_1}_{A_2}$  are preserved.

Let us check the condition (2). Similar the proof of Theorem \ref{t1_2_8},
$F: \lim\limits_\leftarrow S^{*}(X,\lambda,Y) \rightarrow
C_{\lambda,\mu}(X,Y)$ is a uniform homeomorphism.

Similar the proof of Theorem \ref{t1_3_6}, operations are
preserved.
\end{proof}

\begin{corollary}\label{t1_3_17}{\it Let $X$ be a Hausdorff space,
$\lambda\subseteq 2^X$, $Y$ be a Hausdorff topological ring and
$\mu$ be a natural uniformity on $Y$. Let $e:X_{\lambda}
\rightarrow X$ be a natural condensation of $\lambda$-leader
$X_\lambda$ onto $X$ and for every $A\in\lambda$ and $f\in
C(X_\lambda,Y)$ the set $f^{-1}(e(F))$ is bounded in $Y$. Then
$C_{\lambda,\mu}(X,Y)$ is a subring of the topological ring
$\lim\limits_\leftarrow S^{*}(X,\lambda,Y)$.}
\end{corollary}

 Similar of Theorem \ref{t1_3_13} for Tychonoff spaces,
the statement of Corollary \ref{t1_3_1} can be simplified.

\begin{corollary}\label{t1_3_18} {\it Let $X$ be a Tychonoff space,
$\lambda\subseteq 2^X$, $Y$ be a complete Hausdorff topological
ring and $\mu$ be a natural uniformity on $Y$. If for every
$A\in\lambda$ and $f\in C(X,Y)$ the set $f(A)$ is bounded in $Y$
(i.e., $C_{\lambda,\mu}(X,Y)$ is a topological ring) then the
topological ring $\lim\limits_\leftarrow S^{*}(X,\lambda,Y)$ is
the completion of $C_{\lambda,\mu}(X,Y)$.}
\end{corollary}

\begin{proof}
It is enough to prove that
$C_{e^{-1}(\lambda),\mu}(X_{\tau\lambda},Y)$ is a topological
ring, and apply Theorems \ref{t1_3_7} and \ref{t1_3_16}. Note that
a uniformity of the space $C_{\lambda,\mu}(X,Y)$ is a uniformity
of ring $C_{\lambda,\mu}(X,Y)$. Recall that the completion of a
topological ring is a topological ring (Proposition 6 in
\cite{Bur69}).
\end{proof}

\begin{corollary}\label{t1_3_19}{\it Let $X$ be a Tychonoff space, $Y$ be a complete Hausdorff topological ring with uniformity
$\mu$, $\lambda$ be a cover of $X$ and $C_{\lambda,\mu}(X,Y)$ be a
topological ring. If $X_{\tau\lambda}$ is a $\lambda_f$-leader of
$X$ and $e:X_{\tau\lambda} \rightarrow X$ is a natural
condensation then for every $B \in e^{-1}(\lambda)$ and
$f:X_{\tau\lambda} \rightarrow Y$ the set $f(B)$ is bounded in
$Y$. In particular, if $\lambda$ is a family of bounded subsets of
$X$ then $e^{-1}(\lambda)$ is a family of bounded subsets of
$X_{\tau\lambda}$.}
\end{corollary}

\begin{proof} In the proof of Corollary
\ref{t1_3_18}, we proved that
$C_{e^{-1}(\lambda),\mu}(X_{\tau\lambda},Y)$ is a topological
ring. Remain to apply Theorem \ref{t1_3_9}.
\end{proof}

\begin{remark} For a Tychonoff space $X$, the condition for the
boundedness of $f(e^{-1}(A))$ for every $A\in\lambda$ and $f\in
C(X_{\tau\lambda},Y)$ can be replaced by the condition that the
space $C_{\lambda,\mu}(X,Y)$ is a ring.
\end{remark}

\subsection{Topological vector spaces}

The last algebraic structure study in this section is when
$C(X,Y)$ is a vector space over the field of real numbers.

If $X$ is a topological space, $Y$ is a topological vector space
then $C(X,Y)$ is a vector space.

Let $\mu$ be a natural uniformity on $Y$ (a uniformity of a
topological vector space is called {\it natural} if it is
uniformity of it's of additive group). The following theorem get
the answer on the question: when $C_{\lambda,\mu}(X,Y)$ is a
topological vector space.

\medskip

Recall that $B\in Y$ is said {\it bounded}  (in a topological
vector space) if for every neighborhood $U$ of zero in $Y$ there
is $\alpha \in \mathbb{R}$ such that $\alpha U \in B$
(\cite{Bur75}).

\begin{theorem}\label{t1_3_20} Let $Y$ be a topological vector space,
$\mu$ be a natural uniformity on $Y$, $X$ be a topological space
and $\lambda$ be a cover of $X$. Then $C_{\lambda,\mu}(X,Y)$ is a
topological vector space if and only if $A$ is $Y$-bounded for
every $A\in\lambda$.
\end{theorem}

\begin{proof}

By Propositions \ref{t1_3_1} and \ref{t1_3_2},
$C_{\lambda,\mu}(X,Y)$ is a topological group with a base of
neighborhoods  $\beta_c=\{U(A,V):A\in\lambda,V \in \beta\}$ of
zero where $\beta$ is a base of neighborhoods of zero in $Y$.

It is enough to prove that the following four conditions are true
if and only if any continuous image of an element of $\lambda$
is bounded in $Y$.

(1) For every $W\in\beta_c$ there is $W_1 \in \beta_c$ such that
$W_1+W_1 \in W$.

(2) Every $W\in\beta_c$ is absorbing (i.e. $\forall f\in
C(X,Y)\exists\alpha  \in R:f \in \alpha W$).

(3) Every $W\in\beta_c$ is balanced set (i.e. $\forall\alpha \in
R:\mid\alpha\mid\le 1\Leftrightarrow \alpha W \in W$).

(4) For every $W\in\beta_c$ and  $\alpha \in \mathbb{R}$
$(\alpha\neq 0)$, $\alpha W\in\beta_c$.

By Proposition \ref{t1_3_1}, the condition (1) holds for any
family $\lambda$.

The condition (4) is true for any family $\lambda$. Indeed, let
$W=[A,V]$ belongs $\beta_c$. Clearly $\alpha W=[A,\alpha V]$ and
it belongs to $\beta_c$.

Let us check the condition (3).  Let $W=[A,V] \in \beta_c$, then $\alpha
W=[A,\alpha V] \in [A,V]$ where $\mid\alpha\mid\le 1$. Thus, the
condition (3) is true for any family $\lambda$.

Let $\lambda$ be a family of $Y$-bounded subsets of $X$, i.e.
whenever $A\in\lambda$ and $f\in C(X,Y)$ the set $f(A)$ is bounded
in $Y$. Let $g\in C(X,Y)$. We show that any neighborhood
$W=[A,V]\in \beta_c$ absorbing to $g$. Indeed, $g(A)$ is bounded
in $Y$, hence, there is $\alpha \in R$ such that $g(A) \in \alpha
$V. Then $g \in [A,\alpha V]=\alpha [A,V]$, i.e. $g$ absorbing by
$W$.

Conversely, assume that every $W\in\beta_c$ is absorbing and there
are $A\in\lambda$ and $f\in C(X,Y)$ such that $f(A)$ is not
bounded in $Y$. Then for every neighborhood $V$ of zero in $Y$,
$W=[A,V]$ is not absorbing to $f$.

Thus, we proved that condition (2) is true if and only if any
continuous image of an element of $\lambda$ is bounded in $Y$
which completes the proof of theorem.
\end{proof}

\begin{corollary}\label{t1_3_21}
{\it Let $X$ be a topological space, $\lambda$ be a family of
bounded (functionally bounded) subsets of $X$. Then
$C_{\lambda,\rho}(X)$ is a topological vector space.}
\end{corollary}

\begin{corollary}\label{t1_3_22}
{\it Let $X$ be a topological space, $Y$ be a Hausdorff
topological vector space. Then $C_c(X,Y)$ is a topological vector
space.}
\end{corollary}

\begin{corollary}\label{t1_3_23}
{\it Let $X$ be a topological space. Then $C_{\rho}(X)$ is a
topological vector space if and only if $X$ is pseudocompact.}
\end{corollary}

\begin{theorem}\label{t1_3_24}
Let $X$ be a Hausdorff space (resp. Tychonoff space), $Y$ be a
Hausdorff topological vector space, $\mu$ be a natural uniformity
on $Y$.  If a family $\lambda$ of subsets of $X$ such that for any
continuous function $f:X_\lambda \rightarrow Y$ (resp.
$f:X_{\tau\lambda} \rightarrow Y)$ from a $\lambda$-leader (resp.
$\lambda_f$-leader) into $Y$ and each $A \in e^{-1}(\lambda)$ the
set $f(A)$ is bounded in $Y$ where $e:X_\lambda \rightarrow X$
(resp. $e:X_{\tau\lambda} \rightarrow X$) is a natural
condensation of $\lambda$-leader $X_\lambda$ (resp.
$\lambda_f$-leader $X_{\tau\lambda})$ onto $X$, then
$e^\#:C_{\lambda,\mu}(X,Y) \to
C_{e^{-1}(\lambda),\mu}(X_\lambda,Y)$ (resp.
$e^\#:C_{\lambda,\mu}(X,Y) \to
C_{e^{-1}(\lambda),\mu}(X_{\tau\lambda},Y))$ is a linear
embedding.
\end{theorem}

\begin{proof} By Theorem \ref{t1_3_20}, the topologies of spaces
$C_{\lambda,\mu}(X,Y),\linebreak
C_{e^{-1}(\lambda),\mu}(X_\lambda,Y)$ and
$C_{e^{-1}(\lambda),\mu}(X_{\tau\lambda},Y)$ compatible with
linear-structured of space. By Theorem \ref{t1_2_3}, $e^\#:
C_{\lambda,\mu}(X,Y) \to C_{e^{-1}(\lambda),\mu}(X_\lambda,Y)$ and
$e^\#:C_{\lambda,\mu}(X,Y) \to
C_{e^{-1}(\lambda),\mu}(X_{\tau\lambda},Y)$ are uniform embedding.
To complete the proof, it suffices to show that the mapping $e^\#$
is linear. Indeed, $e^\#(\alpha_1f$+$\alpha_2g)(x)=(\alpha_
1f$+$\alpha_2g)(e(x))=\alpha_1f(e(x))$+$\alpha_2g(e(x))=
\alpha_1e^\#(f)(x)$+$\alpha_2e^\#(g)(x)$, i.e.
$e^\#(\alpha_1f$+$\alpha_2g)$=$\alpha_1e^\#(f)$+$\alpha_2e^\#(g)$.
\end{proof}

Similar the case of topological rings, Theorem \ref{t1_3_24} has
the following corollaries.

\begin{corollary}\label{t1_3_25}
{\it Let $X$ be a Hausdorff space, $Y$ be a Hausdorff topological
vector space, then the topological vector space $C_c(X,Y)$ can be
viewed as a linear subspace of the space $C_c(X_{k},Y)$ where
$X_{k}$ is a $k$-leader of $X$.}
\end{corollary}

\begin{corollary}\label{t1_3_26}
{\it Let $X$ be a Hausdorff space, $\lambda$ be a family of
pseudocompact subsets of $X$ then the topological vector space
$C_{\lambda,\rho}(X)$ can be viewed as a linear subspace of the
space $C_{e^{-1}(\lambda),\rho}(X_{ps})$.}
\end{corollary}

\begin{theorem}\label{t1_3_27}
Let $X$ be a Hausdorff space, $\lambda$ be a cover of $X$ and $X$
be a $\lambda_f$-space. Let $Y$ be a Hausdorff topological vector
space and for every $f:X \rightarrow Y$ and every $A\in\lambda$,
the set $f(A)$ is bounded in $Y$. If $\mu$ is a natural uniformity
on $Y$ then $C_{\lambda,\mu}(X,Y)$ is isomorphic to an inverse
limits of the system
$S^{*}(X,\lambda,Y)=\{C^{*}_\mu(\overline{A},Y),\pi^{A_1}_{A_2},\lambda\}$.
\end{theorem}

\begin{proof} Let us check that the mappings $\pi^{A_1}_{A_2}$ are continuous and linear.
The continuity of the mappings $\pi^{A_1}_{A_2}$ follows from
Lemma \ref{t1_2_6}.  Let $h_1$ and $h_2$ belongs $C^{*}_{\mu
}([A],Y)$. Then
$\pi^{A_1}_{A_2}(\alpha_1f_1+\alpha_2f_2)=(\alpha_1f_1+\alpha_2f_2)\vert_{[A_2]}=
\alpha_1f_1\vert_{[A_2]}+\alpha_2f_2\vert_{[A_2]}=\alpha_1\pi^{A_1}_{A_2}(f_1)+\alpha_2\pi^{A_1}_{A_2}(f_2)$.
Thus, we proved that $S^{*}(X,\lambda,Y)$ is a inverse limit of
topological vector spaces and $\lim\limits_\leftarrow
S^{*}(X,\lambda,Y)$ is a topological vector space.

Let us prove that $F:\lim\limits_\leftarrow
S^*(X,\lambda,Y)\to C_{\lambda,\mu}(X,Y)$, constructed in the
proof of Theorem \ref{t1_2_7}, is a linear homeomorphism. By
Theorem \ref{t1_2_8}, $F$ is a homeomorphism. Claim that $F$ is
linear. Let $f_1=\{f^1_A\}_{A\in\lambda}$,
$f_2=\{f^{2}_A\}_{A\in\lambda}\in \lim\limits_\leftarrow
S^{*}(X,\lambda,Y)$. Then
$\alpha_1f_1+\alpha_2f_2=\{\alpha_1f^{1}_A+
\alpha_2f^{2}_A\}_{A\in\lambda}$. Clearly
$\alpha_{1}\bigtriangledown_{A\in\lambda}f^{1}_A(x)+
\alpha_{2}\bigtriangledown_{A\in\lambda}f^{2}_A(x)=\bigtriangledown_{A\in\lambda}
\alpha_1f^{1}_A(x)+
\bigtriangledown_{A\in\lambda}\alpha_2f^{2}_A(x)=\bigtriangledown_{A\in\lambda}(\alpha_1f^{1}_A+
\alpha_2f^{2}_A)(x)$, hence, $F$ is linear.
\end{proof}

\begin{corollary}\label{t1_3_28}
{\it Let $X$ be a Hausdorff space, $\lambda$ be a cover of $X$,
$Y$ be a Hausdorff topological vector space and $\mu$ be a natural
uniformity on $Y$.  Let $e:X_\lambda \rightarrow Y$ be a natural
condensation from $X_\lambda$ onto $X$ and for every $A\in\lambda$
and $f\in C(X_\lambda,Y)$ the set $f(e^{-1}(F))$ is bounded subset
of $Y$.  Then $C_{\lambda,\mu}(X,Y)$ is a linear subspace of the
space $\lim\limits_\leftarrow S^{*}(X,\lambda,Y)$}.
\end{corollary}

Similar the case of topological rings, the condition that
$f(e^{-1}(A))$ is bounded in $Y$ for a Tychonoff space $X$ can be
reduced to the condition that $C_{\lambda,\mu}(X,Y)$ be a
topological vector space (see Corollary \ref{t1_3_30}).

\begin{corollary}\label{t1_3_29}
{\it Let $X$ be a Tychonoff space, $\lambda$ be a cover of $X$,
$Y$ be a complete Hausdorff topological vector space and $\mu$ be
a natural uniformity on $Y$. If for any $A\in\lambda$ and $f\in
C(X,Y)$ the set $f(A)$ is bounded in $Y$ then the topological
vector space $\lim\limits_\leftarrow S^{*}(X,\lambda,Y)$ is a
completion of the topological vector space
$C_{\lambda,\mu}(X,Y)$}.
\end{corollary}

\begin{proof}  It is enough to prove that
$C_{e^{-1}(\lambda),\mu}(X_{\tau\lambda},Y)$ is a topological
vector space and that it is the completion of
$C_{\lambda,\mu}(X,Y)$. By Theorem \ref{t1_2_11},
the space $C_{e^{-1}(\lambda),\mu}(X_{\tau\lambda},Y)$ is the completion of
$C_{\lambda,\mu}(X,Y)$. By Remark \ref{t1_3_7}, the uniformity of
the space $C_{\lambda,\mu}(X,Y)$ is a natural uniformity of the
topological vector space $C_{\lambda,\mu}(X,Y)$, and the
completion of a topological vector space by the natural uniformity
is a topological vector space (Ch.1,1 in \cite{Bur69}). Hence,
$C_{e^{-1}(\lambda),\mu}(X_{\tau\lambda},Y)$ is a topological
vector space.
\end{proof}

\begin{corollary}\label{t1_3_30}
{\it Let $X$ be a Tychonoff space, $Y$ be a complete Hausdorff
topological vector space, $\lambda$ be a cover of $X$, and for
each $A\in\lambda$ and $f\in C(X,Y)$ the set $f(A)$ is bounded in
$Y$. If $X_{\tau\lambda}$ is a $\lambda _f$-leader of $X$ and
$e:X_{\tau\lambda} \rightarrow X$ is a natural condensation from
$X_{\tau\lambda}$ onto $X$ then for each $B\in e^{-1}(\lambda)$
and a continuous function $g:X_{\tau\lambda} \rightarrow Y$ the
set $g(B)$ is bounded in $Y$. In particular, if $\lambda$ is a
family of bounded sets then $e^{-1}(\lambda)$ is a family of
bounded sets, too}.
\end{corollary}

\begin{proof} In the proof of Corollary \ref{t1_3_29}, in particular, we proved that
the space $C_{e^{-1}(\lambda),\mu}(X_{\tau\lambda},Y)$ is a topological
vector space. It remains to apply Theorem \ref{t1_3_20}.
\end{proof}

\section{Open questions}

For different families $\lambda$ and $\gamma$ of uniformities of
spaces $C_{\lambda,\mu}(X,Y)$ and $C_{\gamma,\mu}(X,Y)$ can be
equivalent for any uniform space $Y$. In particular, if each
element of the family $\gamma$ is a subset of the set
$\overline{\bigcup\limits_{i=1}^n B_i}$ where $B_i\in\lambda$ and
each element of $\lambda$ is a subset of
$\overline{\bigcup\limits_{i=1}^n A_i}$, where $A_i\in\gamma$,
then uniformities of $C_{\lambda,\mu}(X,Y)$ and
$C_{\gamma,\mu}(X,Y)$ are equivalent (\cite{Bur75b}).

\medskip

{\bf Question 1}. Let $X$ be a Tychonoff space and
$\lambda\subseteq 2^X$. For which equivalent uniformities $\mu_1$
and $\mu_2$ on $Y$ are the uniformities on the spaces $C_{\lambda,
\mu_1}(X,Y)$ and $C_{\lambda, \mu_2}(X,Y)$ equivalent? In
particular, what if $X$ is compact?

\medskip

{\bf Question 2}. Let $X$ be a Tychonoff space and
$\lambda\subseteq 2^X$. For which equivalent metrics $\rho_1$ and
$\rho_2$ on $Y$ are the spaces $C_{\lambda, \rho_1}(X,Y)$ and
$C_{\lambda, \rho_2}(X,Y)$ homeomorphic? In particular, what if
$\lambda$ consists of compact subsets of $X$?

\medskip

{\bf Question 3}. For which property $\mathcal{P}$ will a family
$\lambda\subseteq 2^X$ whose elements have property $\mathcal{P}$
be a natural cover of $X$? What, if $\mathcal{P}$ is a property of
the type of boundedness? In particular, what if $\lambda$ consists
of pseudocompact subsets of $X$?

\medskip

{\bf Acknowledgements.} The authors would like to thank the
referee for careful reading and valuable comments.

\bibliographystyle{model1a-num-names}
\bibliography{<your-bib-database>}

\begin{thebibliography}{10}

\bibitem{alno}
M.I. Al'perin, S.E. Nokhrin, Topologies on spaces of continuous
functions and embeddings, Trudy Inst. Mat. i Mekh. Uro RAN, 3,
1995, 65--73 (in Russian).


\bibitem{AreDug}
R. Arens, J. Dugundji, Topologies for function spaces, Pac. J.
Math., 1, 1951, 5--31.




\bibitem{arh}
A.V.~Arhangel'skii, Bicompact sets and the topology of spaces, (in
Russian) Trudy Moskov. Mat. Ob\v{s}\v{c}. 13, 1965, 3--55.

\bibitem{Arh89b}
A.V.~Arhangel'skii, Topological function spaces, Ed.MGU: 1989, 222
p.

\bibitem{BecNarSuf}
E. Beckenstein, L. Narici, C. Suffel, Topological Algebras,
Amsterdam etc.: North-Holland, 1977, 370 p.


\bibitem{bd}
A. Bouchair, I. Dekkar, Necessary and sufficient conditions
for admissible set open topologies, Topology and its
Applications, 256:1, 2019, 198--207.



\bibitem{bk}
A. Bouchair, S. Kelaiaia, Some results on $C(X)$ with set open
topology, Math. Rep. (Bucur.) 17(67):2, 2015, 167--182.


\bibitem{Bur69}
N. Bourbaki, General Topology: Topological group, 1969, 392
p.


\bibitem{Bur75}
N. Bourbaki, Topological vector space, 1975, 410 p.


\bibitem{Bur75b}
N. Bourbaki, Elements of Mathematics, General Topology,
(Chapters 1-4), Springer-Verlag, Berlin and New York, 1989.


\bibitem{coh}
D.E. Cohen, Spaces with weak topology, Quart. J. Math.
Oxford. 5, 1954, 77--80.


\bibitem{Eng}
R. Engelking, General Topology, Revised and completed
edition. Heldermann Verlag Berlin, 1989.


\bibitem{FedFil}
V.V. Fedorchuk, V.V. Filipov, General topology, Ed. MGU,
1988, 250p. (in Russian).


\bibitem{fox}
R.H. Fox, On topology for function spaces, Bull. Amer. Math. Soc.,
Vol. 51, 1945, 429--432.



\bibitem{fran}
S.P. Franklin, Natural Covers, Comp. Math. 21, 1969, 253--261.


\bibitem{gal}
D. Gale, Compact sets of functions and function rings, Proc.
Amer. Math. Soc. 1, 1950, 303--308.

\bibitem{hbk}
L. Harkat, A. Bouchair, S. Kelaiaia, Comparison of some set
open and uniform topologies and some properties of the restriction
maps, Hacet. J. Math. Stat., 48:1, 2019, 17--27.


\bibitem{Is}
T. Ishii, The Tychonoff functor and related topics, in: Topics in
General Topology, K. Morita and J. Nagata, eds., North- Holland,
(1989) 203–243.


\bibitem{kawi}
S.M. Karnik, S. Willard, Natural covers and R-quotient maps,
Canad. Math. Bull. 25(4), 1982, 456--461.

\bibitem{kell}
J.L. Kelly, General Topology, New York etc.: D.van Nostrand
Co., 1955.

\bibitem{kundu}
S. Kundu, A.B. Raha, The bounded-open topology and its
relativies, Rend. Istit. Mat. Univ. Trieste, 27, 1995, 61--77.

\bibitem{kundu2}
S. Kundu, R.A. McCoy, Weak and support-open topologies on $C(X)$,
Rocky Mountain J. Math. 25:2, 1995, 715--732.

\bibitem{MckNta}
R.A. McCoy, I. Ntantu, Topological Properties of Spaces of
Continuous Functions, LNM 1315, Springer-Verlag, Berlin, 1988.

\bibitem{nohos}
S.E. Nokhrin S.E., A.V. Osipov, On the coincidence of
set-open and uniform topologies, Proceedings of the Steklov
Institute of Mathematics, 267:1, 2009, 184-191.

\bibitem{Ok}
S. Oka,  The Tychonoff functor and product spaces, Proc.
Japan Acad. 54 (1978) 97–100.

\bibitem{os1}
A.V. Osipov, Weakly Set-Open Topology, Proceedings of the
Steklov Institute of Mathematics, 2:3, 2010, 120–141.

\bibitem{os2}
A.V. Osipov, The Set-Open Topology, Topology Proceedings, 37,
2011, 205--217.

\bibitem{os3}
A.V. Osipov, Properties of the $C$-compact-open topology on a
function space, Trudy Inst. Mat. i Mekh. Uro RAN, 17:4, 2011,
258--277 (in Russian).

\bibitem{os4}
A.V. Osipov A.V., Topological-algebraic properties of
function spaces with set-open topologies, Topology and its
Applications, 159:3, 2012, 800--805.


\bibitem{os5}
A.V. Osipov A.V., The $C$-compact-open topology on function
spaces, Topology and its Applications, 159:13, 2012, 3059--3066.

\bibitem{os6}
A.V. Osipov, Algebraic structures on the space of continuous
maps, Vestn. Tomsk. Gos. Univ. Mat. Mekh., 1:170, 2012, 47--53
(in Russian).

\bibitem{os8}
A.V. Osipov,  Topological vector space of continuous functions
with the weak set-open topology, Proceedings of International
Conference on Topology and its Applications ICTA2011, Cambridge
Scientific Publishers, 2012, 257--264.

\bibitem{os9}
A.V. Osipov, On the completeness properties of the
$C$-compact-open topology on $C(X)$,  Ural Mathematical Journal,
1:1(1), 2015, 61--67.

\bibitem{os10}
A.V. Osipov, Uniformity of uniform convergence on the family
of sets, Topology Proceedings, 50, 2017, 79--86.

\bibitem{os11}
A.V. Osipov, Group structures of function spaces with the
set-open topology, Siberian Electronic Mathematical Reports, 14,
2017, 1440--1446.

\bibitem{os12}
A.V. Osipov, W.K. Alqurashi, L.A. Khan,  Set-open Topologies
on Function Spaces, Applied General Topology, 19:1, 2018, 55--64.

\bibitem{tych}
A. Tychonoff, Uber jener Funktionenraum, Math. Ann., Bd.3,
1935, 762--766.



\end{thebibliography}

\end{document}